\documentclass[10pt,a4paper]{article}
\usepackage{amssymb,amsmath}
\usepackage{amsthm}
\usepackage{bbm}
\usepackage{stmaryrd}
\usepackage{hyperref}
\usepackage[usenames,dvipsnames]{xcolor}
\usepackage{color}
\usepackage{graphics}

\input xy
\xyoption{all} 

\numberwithin{equation}{section}

\newtheorem{theorem}{Theorem}[section]

\newtheorem{proposition}[theorem]{Proposition}

\theoremstyle{definition}
\newtheorem{definition}[theorem]{Definition}
\newtheorem{remark}[theorem]{Remark}
\newtheorem{example}[theorem]{Example}
\newtheorem{coexample}[theorem]{Counter Example}

\newcommand{\Id}{\mathbbmss{1}}

\newcommand{\rmd}{\textnormal{d}}

\newcommand{\rmh}{\textnormal{h}}

\DeclareMathOperator{\w}{w}

\DeclareMathOperator{\Hom}{Hom}

\DeclareMathOperator{\GGL}{GGL}
\DeclareMathOperator{\FGL}{FGL}

\newcommand{\catname}[1]{\textnormal{\texttt{#1}}}

\font\black=cmbx10 \font\sblack=cmbx7 \font\ssblack=cmbx5 \font\blackital=cmmib10  \skewchar\blackital='177
\font\sblackital=cmmib7 \skewchar\sblackital='177 \font\ssblackital=cmmib5 \skewchar\ssblackital='177
\font\sanss=cmss10 \font\ssanss=cmss8 
\font\sssanss=cmss8 scaled 600 \font\blackboard=msbm10 \font\sblackboard=msbm7 \font\ssblackboard=msbm5
\font\caligr=eusm10 \font\scaligr=eusm7 \font\sscaligr=eusm5  \font\fraktur=eufm10
\font\sfraktur=eufm7 \font\ssfraktur=eufm5 
\font\bsymb=cmsy10 scaled\magstep2
\def\all#1{\setbox0=\hbox{\lower1.5pt\hbox{\bsymb
       \char"38}}\setbox1=\hbox{$_{#1}$} \box0\lower2pt\box1\;}
\def\exi#1{\setbox0=\hbox{\lower1.5pt\hbox{\bsymb \char"39}}
       \setbox1=\hbox{$_{#1}$} \box0\lower2pt\box1\;}

\def\tx#1{{\fam0\relax#1}}

\newfam\bifam
\textfont\bifam=\blackital \scriptfont\bifam=\sblackital \scriptscriptfont\bifam=\ssblackital

\newfam\blfam
\textfont\blfam=\black \scriptfont\blfam=\sblack \scriptscriptfont\blfam=\ssblack

\newfam\bbfam
\textfont\bbfam=\blackboard \scriptfont\bbfam=\sblackboard \scriptscriptfont\bbfam=\ssblackboard

\newfam\ssfam
\textfont\ssfam=\sanss \scriptfont\ssfam=\ssanss \scriptscriptfont\ssfam=\sssanss
\def\sss#1{{\fam\ssfam\relax#1}}

\newfam\clfam
\textfont\clfam=\caligr \scriptfont\clfam=\scaligr \scriptscriptfont\clfam=\sscaligr

\newfam\frfam
\textfont\frfam=\fraktur \scriptfont\frfam=\sfraktur \scriptscriptfont\frfam=\ssfraktur

\def\hpb#1{\setbox0=\hbox{${#1}$}
    \copy0 \kern-\wd0 \kern.2pt \box0}
\def\vpb#1{\setbox0=\hbox{${#1}$}
    \copy0 \kern-\wd0 \raise.08pt \box0}

\def\pmb#1{\setbox0\hbox{${#1}$} \copy0 \kern-\wd0 \kern.2pt \box0}
\def\pmbb#1{\setbox0\hbox{${#1}$} \copy0 \kern-\wd0
      \kern.2pt \copy0 \kern-\wd0 \kern.2pt \box0}
\def\pmbbb#1{\setbox0\hbox{${#1}$} \copy0 \kern-\wd0
      \kern.2pt \copy0 \kern-\wd0 \kern.2pt
    \copy0 \kern-\wd0 \kern.2pt \box0}
\def\pmxb#1{\setbox0\hbox{${#1}$} \copy0 \kern-\wd0
      \kern.2pt \copy0 \kern-\wd0 \kern.2pt
      \copy0 \kern-\wd0 \kern.2pt \copy0 \kern-\wd0 \kern.2pt \box0}
\def\pmxbb#1{\setbox0\hbox{${#1}$} \copy0 \kern-\wd0 \kern.2pt
      \copy0 \kern-\wd0 \kern.2pt
      \copy0 \kern-\wd0 \kern.2pt \copy0 \kern-\wd0 \kern.2pt
      \copy0 \kern-\wd0 \kern.2pt \box0}

\mathchardef\za="710B  
\mathchardef\zb="710C  
\mathchardef\zg="710D  
\mathchardef\zd="710E  
\mathchardef\zve="710F 
\mathchardef\zz="7110  
\mathchardef\zh="7111  
\mathchardef\zvy="7112 
\mathchardef\zi="7113  
\mathchardef\zk="7114  
\mathchardef\zl="7115  
\mathchardef\zm="7116  
\mathchardef\zn="7117  
\mathchardef\zx="7118  
\mathchardef\zp="7119  
\mathchardef\zr="711A  
\mathchardef\zs="711B  
\mathchardef\zt="711C  
\mathchardef\zu="711D  
\mathchardef\zvf="711E 
\mathchardef\zq="711F  
\mathchardef\zc="7120  
\mathchardef\zw="7121  
\mathchardef\ze="7122  
\mathchardef\zy="7123  
\mathchardef\zf="7124  
\mathchardef\zvr="7125 
\mathchardef\zvs="7126 
\mathchardef\zf="7127  
\mathchardef\zG="7000  
\mathchardef\zD="7001  
\mathchardef\zY="7002  
\mathchardef\zL="7003  
\mathchardef\zX="7004  
\mathchardef\zP="7005  
\mathchardef\zS="7006  
\mathchardef\zU="7007  
\mathchardef\zF="7008  
\mathchardef\zW="700A  
\mathchardef\zC="7009  

\newcommand{\be}{\begin{equation}}
\newcommand{\ee}{\end{equation}}
\newcommand{\ra}{\rightarrow}

\newcommand{\bea}{\begin{eqnarray}}
\newcommand{\eea}{\end{eqnarray}}
\def\*{{\textstyle *}}
\newcommand{\R}{{\mathbb R}}

\newcommand{\s}{{\textstyle *}}

\newcommand{\ti}{\times}
\newcommand{\A}{{\cal A}}

\def\Hom{\sss{Hom}}

\def\sJ{{\sss J}}

\def\sT{{\sss T}}

\def\sV{{\sss V}}

\def\sv{{\sss v}}

\def\xi{\tx{i}}

\def\cF{\mathcal{F}}

\def\s*{{\scriptstyle *}}

\newcommand{\beas}{\begin{eqnarray*}}
\newcommand{\eeas}{\end{eqnarray*}}

\newcommand{\dd}{\mathbf{d}}
\newcommand{\Gr}{\textnormal{Gr}}
\newcommand{\mbA}{\mathbb{A}}

\def\pLinr{{\mathbf{l}}} 
\def\Linr{{\mathbf{L}}} 

\begin{document}
\bibliographystyle{plain}

\author{\\
        Andrew James Bruce$^1$\\
        Katarzyna Grabowska$^2$\\
        Janusz Grabowski$^3$\\
        \\
        $^1$ {\it Mathematics Research Unit}\\
        {\it University of Luxembourg}\\
         $^2$ {\it Faculty of Physics}\\
         {\it University of Warsaw}\\
         $^3$ {\it Institute of Mathematics}\\
                {\it Polish Academy of Sciences }
               }

\date{\today}
\title{On the Concept of a Filtered Bundle \thanks{Research of KG and JG funded by the  Polish National Science Centre grant under the contract number DEC-2012/06/A/ST1/00256. }}
\maketitle

\begin{abstract}
We present the notion of a \emph{filtered bundle} as a generalisation of a  graded bundle. In particular, we weaken the necessity of the transformation laws for local coordinates to exactly respect the weight of the coordinates by allowing  more general polynomial transformation laws. The key examples of such bundles include affine bundles and various jet bundles, both of which play fundamental r\^{o}les in geometric mechanics and classical field  theory. We also present the notion of \emph{double filtered bundles} which provide natural generalisations of double vector bundles and double affine bundles. Furthermore, we show that the \emph{linearisation}  of a filtered bundle --  which can be seen as  a partial polarisation of the admissible changes of local coordinates  --  is well defined.
\end{abstract}

\begin{small}
\smallskip

\end{small}

\tableofcontents

\section{Introduction}\label{sec:Intro}
Graded bundles, as defined by Grabowski and Rotkiewicz \cite{Grabowski:2012}, offer a very concise and workable notion of a `non-negatively graded manifold'. The general theory of graded manifolds in our understanding was first formulated  by Voronov  in \cite{Voronov:2001qf}, though we   we must also mention the  works of Kontsevich \cite{Kontsevich:2003}, Mehta \cite{Mehta:2006}, Roytenberg \cite{Roytenberg:2002} and \v{S}evera \cite{Severa:2005}, where various notions of a $\mathbb{Z}$-graded supermanifold appear. They play an ever increasing r\^{o}le in pure mathematics and mathematical physics.\par
Loosely, a graded bundle is a fibre bundle for which the base coordinates are assigned a weight of zero and the fibre coordinates are assigned a weight taking non-zero integer values. The transformation laws for the local coordinates are polynomial in non-zero weight coordinates and preserve the weight.\par
There are many  and natural examples of graded bundles, for instance  vector bundles and higher tangent bundles. The latter were the motivation for developing a higher-order Lagrangian formalism for other graded bundles \cite{Bruce:2014b}. Note also that there is a  one-to-one association between vector bundles and graded bundles of degree one, i.e.,  fibre bundles for which we assign a weight of one to the fibre coordinates and a weight of zero to the base coordinates. From the definition, changes of fibre coordinates must be polynomial in the fibre coordinates and preserve the weight, thus they are linear.  However,  it is clear that this notion does not cover all the polynomial bundle structures of interest in geometric mechanics and classical field theory. In particular, jets of sections of a fibre bundle do \emph{not} always form graded bundles, but instead we have a more general filtered transformation law for the local coordinates. That is, while one can still work with coordinates that are assigned a weight, the transformation rules  do not preserve the weight `on the nose', but also include terms of lower weight.   We will in due course define carefully what we will refer to as \emph{filtered bundles}, taking jets  of sections of a fibre bundle as our cardinal example. The degree of a filtered bundle is defined as the highest weight of the homogeneous local coordinates employed on the said filtered bundle. From the polynomial bundle structure the degree is well defined globally. We can as well consider \emph{filtered manifolds} just not requiring that we deal with a fibration with typical $\R^N$-fibers.    \par
Affine bundles also provide simple examples of  filtered bundles: we can assign a weight of one to the fibre coordinates and zero to the base coordinates. However, the transformation laws for the fibre coordinates are not strictly linear  and so the weight is not exactly preserved. Thus, affine bundles are \emph{not} vector bundles, however one  always has the associated model vector bundle. \par
In a similar way, given any filtered bundle it is possible to canonically associate with it a graded bundle, and then via splitting this graded bundle we arrive at a Whitney sum of vector bundles (i.e., a split graded bundle). More than this, we have an isomorphism between the filtered bundle and the associated graded bundle, however this isomorphism is nearly never canonical. One should of course be reminded of the famous Bachelor--Gaw\c{e}dzki theorem on the splitting of (real) supermanifolds. However, the situation for filtered bundles is far less technical. In a natural sense, filtered bundles are related to graded bundles in a way similar to how filtered algebras are related to graded ones. \par
As we shall show, associated with a filtered bundle is a very particular tower of affine bundles. This is very similar to the case for graded bundles. One should of course keep in mind the tower of affine bundles associated with jet bundles.  It is clear that not all towers of affine bundles arise from the underlying structure of a filtered bundle. In this paper we uncover a simple characterisation  of towers of affine bundles that are associated with a filtered bundle. Our methodology using filtered bundles  in part simplifies the identification of jet bundles using an inductive process. In particular, we have clear criterion for identifying filtered bundles, or really that a given tower of affine bundles is what we will call \emph{filterisable}. That is, via a refinement of the bundle atlas we obtain a filtered bundle.  \par
It is quite natural to consider \emph{multiple filtered bundles} as polynomial bundles with several compatible filtered structures. A very specific example is that of filtered-linear bundles, where we have a filtered structure, of degree $k$ say, and a linear or vector bundle structure that are compatible in a sense we will make clear.   Natural examples of filtered-linear structures include tangent bundles of filtered bundles. More general multiple filtered bundles include jet bundles of filtered bundles, for example. The reader should of course keep in mind double and $k$-fold vector bundles. \par
The notion of the linearisation of a graded bundle was first explored in \cite{Bruce:2014}, and is possible due to the non-negatively graded nature of homogeneous local coordinates and their admissible transformation laws. It turns out that the \emph{linearisation of a filtered bundle} is also well defined and enjoys similar functorial properties as the linearisation of a graded bundle. In particular, we have a functor from the category of filtered bundles of degree $k$ to the category of filtered-linear bundles of degree $k-1$, we will make these notions precise in due course. Jets of sections of a fibre bundle give   natural and well studied examples of filtered bundles. However, to our knowledge, the notion of linearising a jet bundle has not appeared in previous literature.  The consequences and applications of  linearisation in the context of jet bundles applied to geometric mechanics and classical field theory are to date  completely unexplored.\par
We remark that in physics one has to confront affine structures in order to obtain frame (gauge) independent theories, see for example \cite{Grabowska:2005,Popescu:2005}. Passing from  affine-objects to vector-objects is interpreted as fixing some observer or fixing a gauge. Double vector bundles have also proven fundamental to geometric mechanics, and in a similar vain double affine bundles have been employed in this setting, see \cite{Grabowski:2010}. Moreover,  jet bundles are the natural geometric framework for discussing partial differential equations and classical field theories.   The langauge of filtered bundles may well prove useful in this context. However, we leave the applications of filtered and multi-filtered structures for potential future publications.\par
\smallskip
\noindent \textbf{Arrangement:} In Section \ref{sec:preliminaries} we remind the reader of the definions of a polynomial bundle and a graded bundle. In that section we define weighted polynomial algebras as these will be essential in defining filtered bundles. It is in Section \ref{sec:FilteredBundles} that we give the notion of filtered bundles and explore their basic structure. In Section \ref{sec:MultipleFiltered} we define and make intial study of double and multiple filtered bundles. In particular in this section we discuss the linearisation of a filtered bundle.

\section{ Preliminaries}\label{sec:preliminaries}
\subsection{Polynomial bundles}
\begin{definition}[Adapted from \cite{Bertram:2014} and \cite{Voronov:2010}]
A \emph{polynomial bundle} is a fibre bundle $\pi : B \longrightarrow M$ in the category of smooth manifolds with typical fibre $\mathbb{R}^N$ (for some $N$) such that the admissible transformation laws for the fibre coordinates are polynomial in the fibre coordinates.
\end{definition}
Examples of polynomial bundles are numerous and include, vector bundles, graded bundles, various jet bundles and more general Weil bundles (see for example \cite{Kolar:1993}).

\subsection{Graded manifolds and graded bundles}
Let us briefly recall the key components of the theory of graded bundles.   We will consider smooth manifolds explicitly, although the statements in this subsection generalise to the supercase (cf. \cite{Voronov:2001qf}).\par
An important  class of graded manifold are those that carry a non-negative grading. We will require that this grading is associated with a smooth action $\rmh:\R\ti F\to F$ of the monoid $(\R,\cdot)$ of multiplicative reals on a manifold $F$, a \emph{homogeneity structure} in the terminology of \cite{Grabowski:2012}.
This action reduced to $\R_{>0}$ is the one-parameter group of diffeomorphism integrating the \emph{weight vector field}, thus the weight vector field is in this case \emph{h-complete} \cite{Grabowski:2013} and only \emph{non-negative integer weights} are allowed. Thus the algebra $\mathcal{A}(F)\subset C^\infty(F)$ spanned by homogeneous functions is $\mathcal{A}(F) =  \bigoplus_{i \in \mathbb{N}}\mathcal{A}^{i}(F)$, where $\mathcal{A}^{i}(F)$ consists of homogeneous functions of degree $i$. \par
For $t \neq 0$ the action $\rmh_{t}$ is a diffeomorphism of $F$ and, when $t=0$, it is a smooth surjection $\tau=\rmh_0$ onto $F_{0}=M$, with the fibres being diffeomorphic to $\mathbb{R}^{N}$.  Thus, the objects obtained are particular kinds of \emph{polynomial bundles} $\tau:F\to M$, i.e. fibrations which locally look like $U\times\R^N$ and the change of coordinates (for a certain choice of an atlas) are polynomial in $\R^N$. For this reason graded manifolds with non-negative weights \emph{and} h-complete weight vector fields  $\zD$ are also known as \emph{graded bundles} \cite{Grabowski:2012}.
\begin{remark}
A proof of the equivalence of monoid actions of $(\R, \cdot)$ on supermanifolds and non-negatively graded supermanifold can be found in  \cite{Jozwikowski:2016}.
\end{remark}
 \begin{example} If the weight is constrained to be either zero or one, then the weight vector field is precisely a  vector bundle structure on $F$ and will be generally referred  to as an \emph{Euler vector field}.
\end{example}
%
\begin{example}\label{e1}
Consider a manifold $M$ and $\dd=(d_1,\dots,d_k)$, with positive integers $d_i$. The trivial fibration $\tau:M\times\R^\dd\to M$, where
$\R^\dd=\R^{d_1}\times\cdots\times\R^{d_k}$, is canonically a graded bundle with the homogeneity structure $\rmh^\dd$ given by $\rmh_t(x,y)=(x,\rmh_t^\dd(y))$, where
\begin{equation}\label{hsm}
\rmh_t^\dd(y_1,\dots,y_k)=(t\cdot y_1,\dots, t^k\cdot y_k)\,,\quad y_i\in\R^{d_i}\,.
\end{equation}
\end{example}
\begin{theorem}\textsc{(Grabowski-Rotkiewicz \cite{Grabowski:2012})}\label{theorem:1}
Any graded bundle $(F,\rmh)$ is a locally trivial fibration $\tau:F\to M$ with a typical fiber $\R^\dd$, for some $\dd=(d_1,\dots,d_k)$, and the homogeneity structure locally equivalent to the one in Example \ref{e1}, so that the transition functions are graded isomorphisms of $\R^\dd$. In particular, any graded space is diffeomorphically equivalent with $(\R^\dd,\rmh^\dd)$ for some $\dd$.  The algebras $\A(\R^\dd)$ are examples of canonical \emph{graded polynomial algebras (cf. \cite{Grabowski:2016}).}
\end{theorem}
The sequence $\dd=(d_1,\dots,d_k)$ we call the \emph{rank of the graded bundle $F$}.
It follows that on  a general graded bundle, one can always pick an  atlas of $F$  consisting of charts for which we  have homogeneous local coordinates $(x^{A}, y_{w}^{i})$, where $\w(x^{A}) =0$ and  $\w(y_{w}^{i}) = w$ with $1\leq w\leq k$, for some $k \in \mathbb{N}$ known as the \emph{degree}  of the graded bundle.  Note that, according to this definition, a graded bundle of degree $k$ is automatically a graded bundle of degree $l$ for $l\ge k$. However, there is always a \emph{minimal degree}. It will be convenient to group all the coordinates with non-zero weight together, as these form a basis of the function algebra of the graded bundle.  The index  $i$ should be considered as a ``generalised index" running over all the possible weights. The label $w$ in this respect largely redundant, but it will come in very useful when checking the validity of various expressions.  The local changes of coordinates  respect the weight and hence are polynomial for non-zero weight coordinates. A  little more explicitly, the changes of local coordinates are of the form
\begin{eqnarray}\label{eqn:translawsGr}
x^{A'} &=& x^{A'}(x),\\
\nonumber y^{a'}_{w} &=& y^{b}_{w} T_{b}^{\:\: a'}(x) + \sum_{\stackrel{1<n  }{w_{1} + \cdots + w_{n} = w}} \frac{1}{n!}y^{b_{1}}_{w_{1}} \cdots y^{b_{n}}_{w_{n}}T_{b_{n} \cdots b_{1}}^{\:\:\: \:\:\:\:\:a'}(x),
\end{eqnarray}
where $T_{b}^{\:\: a'}$ are invertible and the $T_{b_{n} \cdots b_{1}}^{\:\:\: \:\:\:\:\:a'}$ are symmetric in lower indices. For every $\dd$ we can consider the Lie group $\GGL(\dd)$ of graded isomorphisms of $\R^\dd$, i.e. diffeomorphisms of $\R^\dd$ of the form
\be
y^{a'}_{w} = y^{b}_{w} T_{b}^{\:\: a'} + \sum_{\stackrel{1<n  }{w_{1} + \cdots + w_{n} = w}} \frac{1}{n!}y^{b_{1}}_{w_{1}} \cdots y^{b_{n}}_{w_{n}}T_{b_{n} \cdots b_{1}}^{\:\:\: \:\:\:\:\:a'}\,.
\ee
The group $\GGL(\dd)$ acts tautologically on $\R^\dd$. The proof of the following is completely parallel to that for the case of a vector bundle.
\begin{theorem}
Any graded bundle $F$ of rank $\dd$ is a bundle associated with a principal bundle $P\to M$ of the Lie group $\GGL(\dd)$. The principal bundle $P$ is the frame bundle of $F$, i.e. the bundle of graded isomorphisms of the fibers of $F\to M$ with $\R^\dd$.
\end{theorem}
Importantly, a graded bundle  of degree $k$  admits a sequence of  surjections
\begin{equation}\label{eqn:fibrations}
F=F_{k} \stackrel{\tau^{k}_{k-1}}{\longrightarrow} F_{k-1} \stackrel{\tau^{k-1}_{k-2}}{\longrightarrow}   \cdots \stackrel{\tau^{3}_2}{\longrightarrow} F_{2} \stackrel{\tau^{2}_1}{\longrightarrow}F_{1} \stackrel{\tau^{1}}{\longrightarrow} F_{0} = M,
\end{equation}
\noindent where $F_l$ itself is a graded bundle over $M$ of degree $l$ obtained from the atlas of $F_k$ by removing all coordinates of degree greater than $l$ (see the next paragraph). \par
Note that  $F_{1} \rightarrow M$ is a linear fibration and the other fibrations $F_{l} \rightarrow F_{l-1}$ are affine fibrations in the sense that the changes of local coordinates for the fibres are linear plus  an additional additive terms of appropriate weight.  The model fibres here are $\mathbb{R}^{n}$.\par
A homogeneity structure is said to be \emph{regular} if
\begin{equation}\label{eqn:HomoReg}
\left.\frac{\rmd }{\rmd t}\right|_{t=0}\rmh_{t}(p) = 0   \hspace{25pt}\Rightarrow   \hspace{25pt}  p = \rmh_{0}(p),
\end{equation}
for all points $p \in F$. Moreover, if homogeneity structure is regular then the graded bundle is of degree $1$ and we have a vector bundle structure. The converse is also true and so $\tau^{1}:F_{1} \rightarrow M $ is a vector bundle. It is an amazing fact that all the structure of a vector bundle can be encoded in a regular homogeneity structure, (see \cite{Grabowski:2009}).
\begin{example} The principal canonical example of a graded bundle is the higher tangent bundle $\sT^{k}M$; i.e. the $k$-th jets (at zero) of curves $\gamma: \mathbb{R} \rightarrow M$.
\end{example}
Morphisms between graded bundles necessarily preserve the weight. In other words morphisms relate the respective homogeneity structures, or equivalently morphisms relate the respective weight vector fields. Evidently, morphisms of graded bundles can be composed as standard maps between smooth manifolds and so we obtain the category of graded bundles, which we denote as $\catname{GrB}$. The full subcategory of graded bundles of a given minimal degree $k \in \mathbb{N}$ will be denoted as $\catname{GrB}[k]$.
\subsection{Filtered polynomial algebras}
Recall that a (real) \emph{graded polynomial algebra} is understood as an algebra isomorphic to the polynomial algebra $\A(\R^\dd)$ with the $\mathbb{N}$-gradation induced from the homogeneity structure of $\R^\dd$. Such algebras can be also characterised as commutative real algebras $\A$ freely generated by a finite set of generators
(indeterminants) $\{ X^{I}_{w}\}$, where $I$ is a (multi-)index  and $w \in \mathbb{N}\setminus \{0\}$ is a label that describes the weight:
$$\w(X^{I}_{w}) = w\,.$$
We will refer to these generators (indeterminants) as \emph{homogeneous (or weighted) generators}.  The weight of a \emph{homogeneous monomial} is simply the sum of the weight of the indeterminants that make up the monomial and the concept of a homogeneous polynomial is clear. The algebra $\A$ is therefore canonically graded, $\A=\oplus_{i=0}^\infty \,\A^i$.
\begin{remark}
One could of course consider the homogeneous indeterminants to be supercommutative by assigning them some independent Grassmann parity. For simplicity and applications in mind we consider only strictly commutative indeterminants.
\end{remark}
\begin{definition}\label{def:weightedpoly}
A \emph{weighted polynomial} in the indeterminants $\{ X^{I}_{w}\}$  is a finite sum of monomials in the said weighted indeterminants. The \emph{degree}  of a weighted polynomial is the maximal weight of the monomials (with non-zero coefficients) that make up the weighted polynomial.
\end{definition}
Any polynomial $P \in \mathcal{A}$ is then of the form
$$P =  \sum_{n \in \mathbb{N}} \frac{1}{n!} X^{I_{1}}_{w_{1}}X^{I_{2}}_{w_{2}} \cdots X^{I_{n}}_{w_{n}} P_{I_{n} \cdots I_{2} I_{1}},$$
\noindent where the coefficients are taken to be symmetric, i.e.,
$$P_{I_{n} \cdots L, M, \cdots I_{2} I_{1}} =   P_{I_{n} \cdots M, L, \cdots I_{2} I_{1}}\in\R\,.$$
The degree of any weighted polynomial is given by
$$\textnormal{deg}(P) = \max(w_{1} + w_{2} + \cdots + w_{n}) \in \mathbb{N},$$
over all the non-zero monomials that make up the polynomial.

\begin{example}
Consider the weighted indeterminants $\{ X_{1}, X_{2}\}$ of degree 1 and 2, respectively. The the monomial $X_{1}X_{2}$ is of weight $3$, while the monomial $X_{1}X_1X_{2}=X_{1}^2X_{2}$ is of weight $4$, so that the polynomial $X_{1}X_{1}X_{2}+X_{1}X_{2}$ is of degree 4.\end{example}
In the obvious way the (weighted) degree of a polynomial defines an $\mathbb{N}$-filtration in the algebra of polynomials given by
$$\{0\} \subset \mathcal{A}_{0} \subset \mathcal{A}_{1} \subset \mathcal{A}_{2}\subset\cdots\subset \mathcal{A},$$
where the filtration is given by the degree of the weighted polynomials.  That is we have
$$\mathcal{A} = \bigcup_{n \in \mathbb{N}}\mathcal{A}_{n},$$
\noindent where $\mathcal{A}_{n}$ is the subspace of $\mathcal{A}$ spanned by polynomials of degree $\leq n$,
$$\mathcal{A}_n = \bigoplus_{i=0}^n\mathcal{A}^{i}\,,$$
and
$$\A_n\cdot\A_m\subset\A_{n+m}\,.$$
Any graded polynomial algebra is therefore canonically filtered. Moreover, this filtration is \emph{connected}, i.e., the part $\A_0$ is $\R$.
The pair consisting of the polynomial algebra with this filtration will be called a \emph{filtered polynomial algebra}.
Note that any filtration defines the graded algebra
$$\Gr(\A)=\bigoplus_{i=0}^\infty \A_{i+1}/\A_{i}$$
and it is clear that that the graded algebra associated with the filtration on a graded polynomial algebra $\A$ is isomorphic as a graded algebra (this isomorphism depends on the choice of homogeneous generators) with $\A$.
Actually we have the following (cf. \cite{Grabowski:2016}).
\begin{proposition}\label{prop:graded_generators} Any connected $\mathbb{N}$-filtration $\{\A_i\}_{i\in\mathbb{N}}$ in a polynomial algebra $\A=\R[z^1,\dots,z^N]$ gives a filtered algebra isomorphic to $\A(\R^\dd)$ for some $\dd=(d_1,d_2,\hdots,d_r)$.
\end{proposition}
\begin{proof}
The proof is completely parallel to that of \cite[Proposition 1]{Grabowski:2016}: we successively construct homogeneous generators adding at the $(i+1)$-th step a basis of $\A_{i+1}/(\A_{i+1}\cap\langle\A_i\rangle)$, where $\langle\A_i\rangle$ is the subalgebra generated by $\A_i$. In particular,
$$d_{i+1}=\dim\left(\A_{i+1}/(\A_{i+1}\cap\langle\A_i\rangle)\right)\,.$$
\end{proof}

\section{Filtered  Bundles}\label{sec:FilteredBundles}
\subsection{The definition of a filtered bundle}
Heuristically, we think of a filtered bundle as a graded bundle for which we have now weakened the admissible coordinate transformations to no longer preserve the weight exactly, but instead we can include polynomial terms that are of lower order in weight.
\begin{definition}
A \emph{filtered bundle} $\pi : \mathcal{F} \longrightarrow M$, is polynomial bundle for which the adapted  fibre coordinates are assigned   non-zero weights in $\mathbb{N}$, the base coordinates are assigned weight zero,  and the admissible changes of coordinates respect the degree as weighted polynomials. The  \emph{degree} $k$ of a filtered bundle is the maximum weight of the coordinates, and the \emph{rank} of a filtered bundle of degree $k$ is a sequence $\dd=(d_1,\dots,d_k)$, where $d_i$ is the number of coordinates of degree $i$.
\end{definition}
In clearer terms, this means that on any filtered bundle $\mathcal{F}$ one has adapted local coordinates
$$(\underbrace{x^{a}}_{0}, ~ \underbrace{X_{w}^{I}}_{w}),$$
\noindent $(1 < w \leq  k)$  and the admissible changes of local coordinates are \emph{graded affine}
\begin{align}\label{eqn:tranformlaws}
&x^{a'} = x^{a'}(x), && X_{w}^{I'} =  \sum_{w_{1} + \cdots + w_{n}\leq w} \frac{1}{n!} X^{J_{1}}_{w_{1}} X^{J_{2}}_{w_{2}} \cdots X^{J_{n}}_{w_{n}}T_{J_{n} \cdots J_{2} J_{1}}^{\hspace{30pt} I'}(x).&
\end{align}
Note that the transformation laws are almost identical to that of a graded bundle (\ref{eqn:translawsGr}), except that we now include terms that are \emph{not} homogeneous in weight but of lower order. As the transformation law are invertible, the  matrix in the above describing the linear term is invertible.  This allows us to define the \emph{polynomial function algebra} $\A(\cF)$ as the algebra of  smooth functions on $\cF$ that are polynomial in the coordinates $X_w^I$. It is important to note that this algebra is canonically filtered, $\A(\cF)=\cup_i\A_i(\cF)$.
 \begin{remark}
It is clear directly from the definition that filtered bundle of rank $\dd$ is a bundle associated with a principal bundle with the structure group  $\FGL(\R^\dd)$ of filtered automorphisms of $\R^\dd$, i.e. the group of diffeomorphisms of $\R^\dd$ of the form
\be\label{eqn:tranformlaws1}
X_{w}^{I'} =  \sum_{w_{1} + \cdots + w_{n}\leq w} \frac{1}{n!} X^{J_{1}}_{w_{1}} X^{J_{2}}_{w_{2}} \cdots X^{J_{n}}_{w_{n}}T_{J_{n} \cdots J_{2} J_{1}}^{\hspace{30pt} I'}.
\ee
\end{remark}
 \begin{remark}\label{r11} In other words, we have an $\mathbb{N}$-filtration in the structure sheaf corresponding to a fibration of filtered polynomial algebras (\emph{filtered polynomial algebra bundle}). This fibration is locally trivial, i.e. it is of the form $U\ti\A(\R^\dd)$ for some $\dd$, and the transition functions in our atlas are of the form (\ref{eqn:tranformlaws1}), i.e. consist of automorphisms of filtered polynomial algebras. The space of sections of this bundle gives exactly polynomial function algebra $\A(\cF)$.
\end{remark}
\begin{remark}
We use the nomenclature \emph{filtered bundle} to be consistent with the notion of a \emph{graded bundle}. In particular, we insist that (collectively) the non-zero weight coordinates can take all `values' in $\mathbb{R}^p$ and not just in some subset thereof. The underlying structure of a polynomial bundle is vital in our definition of a filtered bundle, as it is for the definition of a graded bundle.  If we do not insist that the collective fiber coordinates fill the whole of $\R^p$, then we work with a more general geometric object which we can call a \emph{filtered manifold}. Note that  Morimoto \cite{Morimoto:2000} has defined a notion of a \emph{filtered manifold} as a manifold with a particular filtration of distributions (ie., substructures of the tangent bundle), such a notion is distinct from ours. The notion presented here should not be confused with the notion of a \emph{filtered vector bundle}.
\end{remark}
\begin{example}
All graded bundles --  specifically vector bundles and higher tangent bundles -- are examples of filtered bundles.
\end{example}
\begin{example}
Affine bundles $\pi : \mathsf{A} \rightarrow M$ admit  adapted coordinates  $(x^{a}, ~ X^{\alpha})$ and the admissible changes of fibre coordinates are of the form
$$X^{\alpha'} = X^{\beta}T_{\beta}^{\:\:\: \alpha'}(x) + T^{\alpha'}(x).$$
Thus, if we assign weight zero to the base coordinates and weight one to the fibre coordinates we clearly have a filtered bundle of degree one.
\end{example}
\begin{example}\label{exp:jetbundles} The fundamental example is that of a \emph{jet manifold} of sections of a fibre bundle. Let $ \pi: E \rightarrow M$ be a fibre bundle in the category of  smooth manifolds. Then the $k$-th jet manifold of sections $\sJ^{k}E$ is naturally a filtered bundle. A little more specifically, let  $\sJ_x^k E$ denote the set of equivalence classes of sections of $E$ that at $x\in M$ are in contact to order $k$. That is, elements of this set are sections of $E$ for which all the partial derivatives up to and including order $k$ all agree at the point $x$ (see for example \cite[pages 124-125]{Kolar:1993} for details). We then define
$$\sJ^kE := \bigcup_{x \in M} \sJ^k_x E.$$
The filtered bundle structure is best described using natural local coordinates, and it is well known that the admissible changes of coordinates are of the form (\ref{eqn:tranformlaws}). For clarity, let us initially restrict attention to first jets. If we equip $E$ with adapted coordinates $(x^{a},~ y^{\alpha})$, then the first jet $\sJ^1 E$ comes with naturally inherited coordinates $(x^{a},~ y^{\alpha},~z^{\beta}_{b})$. The coordinate transformations here are
\begin{align*}
&x^{a'} = x^{a'}(x), &   y^{\alpha'} = y^{\alpha'}(x,y), &&  z^{\alpha'}_{a'} = \left( \frac{\partial x^{b}}{\partial x^{a'}} \right)\left(\frac{\partial}{\partial x^{b}} + z^{\beta}_{b} \frac{\partial}{\partial y^{\beta}} \right) y^{\alpha'}(x,y).        &
\end{align*}
Clearly these transformation laws are filtered if we make the assignment of weight
$$(\underbrace{x}_{0}, ~ \underbrace{y}_{0} , ~ \underbrace{z}_{1}).$$
The second order jet bundle $\sJ^2E$ similarly comes equipped with natural local coordinates $(x^{a},~ y^{\alpha},~z^{\beta}_{b}, ~ w^\gamma_{cd})$, where the transformation for the coordinates $w$ is
$$w^{\alpha'}_{a'b'} = \left( \frac{\partial x^{c}}{\partial x^{a'}}  \right) \left(  \frac{\partial}{\partial x^c} + z_c^\beta \frac{\partial}{\partial y^\beta}  +  w_{cd}^\beta  \frac{\partial}{\partial z_d^\beta}   \right) z_{b'}^{\alpha'}(x, y,z). $$
To make the filtered structure clear, we can \emph{symbolically} write the transformation laws for $z$ and $w$ as
\begin{align*}
& z' = z \frac{\partial y'}{\partial y} + \frac{\partial y'}{\partial x},\\
& w' = w \frac{\partial y'}{\partial y} +  z z \frac{\partial^2 y'}{\partial y \partial y} + z \frac{\partial^2 y'}{\partial x \partial y} + \frac{\partial^2 y'}{\partial x \partial x},
\end{align*}
by supressing indices and forgetting $x'$. It is then clear that we do indeed have a filtered bundle structure by assigning $w^{\gamma}_{cd}$ weight $2$. Similar considerations show that any $\sJ^kE$ comes with the natural structure of a filtered bundle.
 \end{example}
 \begin{remark}\label{re1}
One has to stress that the jet bundle $\sJ^kE$ as a filtered bundle is the bundle over $E$. In the theory of jet bundles, in the connection with iterated jets, the jet bundle $\sJ^kE$ is considered also as a bundle over $M$. Of course it leads to different understanding of iterated bundles $\sJ^k(\sJ^rE)$. In this paper we will consequently understand
$\sJ^kE$ as a bundle over $E$. Of course one take into account the additional structure of fibration $E\to M$ and consider \emph{filtered bundles over fibered manifolds} but it is outside of the scope of this paper.
\end{remark}
\begin{example}\label{exp:2filtered}
Let us consider a general filtered bundle of degree $2$, which we will denote as $\mathcal{F}_{\{2\}}$. We can employ local coordinates
$$(\underbrace{x^{a}}_{0}, ~ \underbrace{Y^{\alpha}}_{1}, ~\underbrace{Z^{i}}_{2}),$$
\noindent and the admissible changes of fibre coordinates are of the form
\begin{align*}
& Y^{\alpha'} =  Y^{\beta}T_{\beta}^{\:\:\: \alpha'}(x) + T^{\alpha'}(x), && Z^{i'} =  Z^{j}T_{j}^{\:\:\: i'}(x) + \frac{1}{2!} Y^{\alpha} Y^{\beta}T_{\beta \alpha}^{\:\:\:\:\:\: i'}(x) + Y^{\alpha}T_{\alpha}^{\:\:\: i'}(x) + T^{i'}(x). &
\end{align*}
\end{example}

Because we are dealing with `filtered' changes of coordinates rather than `graded' changes of coordinates it is clear that we cannot encode the structure of  a filtered bundle in a smooth action of the monoid of reals:  a \emph{homogeneity structure} in the language of \cite{Grabowski:2013,Grabowski:2009,Grabowski:2012}. This seems to be an obstacle to a clear global understanding of a filtered bundle. \par

Directly from the admissible changes of coordinates we see that  for any filtered bundle of degree $k$  we have, like in the case of a graded bundle, a tower of affine fibrations
\begin{equation}\label{eqn:tower}
\mathcal{F} := \mathcal{F}_{\{k \}} \stackrel{\tau^{k}_{k-1}}{\xrightarrow{\hspace*{20pt}} } \mathcal{F}_{\{k -1\}}  \stackrel{\tau^{k-1}_{k-2}}{\xrightarrow{\hspace*{20pt}} } \mathcal{F}_{\{k -2\}}   \longrightarrow  \cdots \longrightarrow \mathcal{F}_{\{1 \}}  \stackrel{\tau^{1}_{0}}{\xrightarrow{\hspace*{15pt}} } \mathcal{F}_{\{0\}} =: M.
\end{equation}
In particular, we have an affine bundle $\tau^{1}_{0}:\mathcal{F}_{\{1 \}} \rightarrow M$. \par
The weighted polynomial algebra $\mathcal{A}(\mathcal{F})$ has a  module structure over $C^{\infty}(M)$. Let us denote by $\mathcal{A}_{l}(\mathcal{F})$ the subalgebra of $\mathcal{A}(\mathcal{F})$ locally generated by weighted by monomials of weight $ \leq  l$. Then we have the filtration of algebras
$$C^{\infty}(M) = \mathcal{A}_{0}(\mathcal{F}) \subset \mathcal{A}_{1}(\mathcal{F}) \subset \cdots \subset \mathcal{A}_{k}(\mathcal{F})= \mathcal{A}(\mathcal{F})\,,$$
 dual to the tower (\ref{eqn:tower}).
\begin{example} Continuing Example \ref{exp:jetbundles}, it is well known that associated with any jet manifold $\sJ^kE$ of sections of a fibre bundle $\pi: E \rightarrow M$, there is a tower of affine fibrations given by reducing the order of contact of the sections
$$\sJ^kE \stackrel{\tau^{k}_{k-1}}{\xrightarrow{\hspace*{20pt}} } \sJ^{k -1}E  \stackrel{\tau^{k-1}_{k-2}}{\xrightarrow{\hspace*{20pt}} } \sJ^{k -2}E   \longrightarrow  \cdots \longrightarrow \sJ^1E  \stackrel{\tau^{1}_{0}}{\xrightarrow{\hspace*{15pt}} } \sJ^0E := M.$$
\end{example}
An important observation is the following.
\begin{theorem}\label{T0} The vector bundle models of the affine fibrations $\sv(\mathcal{F}_{\{i \}})\rightarrow \mathcal{F}_{\{i -1\}}$ are not arbitrary vector bundles over $\mathcal{F}_{\{i -1\}}$, but  are pull-back bundles of
certain vector bundles $\zs_i:V_i\rightarrow M$, $i=1,\dots,k$.
\be\label{pb}
\sv(\mathcal{F}_{\{i \}})=({\tau^{i}_{0}})^*(V_i)\,.
\ee
\end{theorem}
\begin{proof}
The transition functions in the vector bundle $\sv(\mathcal{F}_{\{i \}})$ are linear parts of the transition functions
\eqref{eqn:tranformlaws}, thus
$$X_{w}^{I'} =  X^{I}_{w}T_{I}^{\hspace{3pt} I'}(x)\,.
$$
This means that we can identify all linear fibers over a fiber of ${\tau^{i}_{0}}:\mathcal{F}_{\{i \}}\rightarrow M$
over $x\in M$ as $V_i(x)$ and view the vector bundle $\sv(\mathcal{F}_{\{i \}})$ as the pull-back bundle of $V_i$.
\end{proof}

\subsection{The category of filtered bundles}
A morphism between filtered bundles is a morphisms of polynomial bundles that respects the  weight of (generally locally non-homogeneous) polynomials. Let us employ local coordinates $(x^{a}, X^{I}_{w})$ and $(y^{\alpha}, Y_{u}^{\Sigma})$ on $\mathcal{F}_{\{k\}}$ and $\mathcal{G}_{\{l\}}$ respectively, then a morphism
$$\phi : \mathcal{F}_{\{k\}} \rightarrow \mathcal{G}_{\{l\}},$$
\noindent is locally of the form
\begin{align}\label{eqn:filtered morphisms}
&\phi^{*}y^{\alpha} = \phi^{\alpha}(x), && \phi^{*}Y^{\Sigma}_{u} =  \sum_{w_{1}+ \cdots w_{n} \leq u}\frac{1}{n!} X^{I_{1}}_{w_{1}} X^{I_{2}}_{w_{2}}\cdots X^{I_{n}}_{w_{n}}\phi^{\Sigma}_{I_{n} \cdots I_{2}I_{1}}(x).&
\end{align}
Evidently filtered bundles and their morphisms form a category, where the morphisms are composed as smooth maps between manifolds.  We will denote this category as $\catname{FilB}$. If we wish to restrict attention to filtered bundles of some given fixed  degree, say $k$, then we write $\catname{FilB}[k]$ for the obvious full subcategory.\par
The category of filtered bundles has some interesting subcategories.
\begin{enumerate}
\item Affine bundles form a full subcategory of $\catname{FilB}$ -- moreover we have an equivalence of categories between $\catname{FilB}[1]$ and the category of affine bundles $\catname{AffB}$.
\item Graded bundles form a subcategory, but not a full subcategory of $\catname{FilB}$ -- local homogeneity is exactly preserved.
\item Jet bundles form a subcategory, but not a full subcategory of $\catname{FilB}$.
\end{enumerate}
\begin{proposition}
Each projection in the tower (\ref{eqn:fibrations}) is a morphism of filtered bundles, and thus, any composition of the projections is a morphism in the category of filtered bundles.
\end{proposition}
\begin{proof}
This follows directly from the local description of filtered bundles.
\end{proof}

 \subsection{Duality}
We have already recognised (see Remark \ref{r11}) filtered polynomial algebra bundles as associated with filtered bundles. Of course, they are infinite-dimensional objects but a corresponding differential geometry can be easily developed and we can view them as objects dual to filtered bundles, exactly like bundles of graded polynomial algebras
are regarded as objects  dual to graded bundles (see \cite{Grabowski:2016}). In other words, the object $\cF^*$  dual to a filtered bundle $\cF$ is the filtered polynomial algebra $\Hom_{f}(\cF,\mbA)$ bundle whose fibers are filtered morphisms $\cF_x\to\mbA$, where $\mbA=\A(\R)=\R[\ze]$ is the filtered algebra of polynomials on $\R$. Note that here we view $\mbA$ as a merely filtered space.\par
As the whole duality theory is completely parallel to that in the case of graded bundles (see \cite{Grabowski:2016}), we skip the details here. Let us only mention that the original filtered bundle $\cF$ can be reconstructed from the filtered polynomial algebra bundle $\mathbb{F}=\cF^*$
as the manifold $\Hom_{f-alg}(\mathbb{F},\mbA)$ of filtered algebra homomorphisms $\mathbb{F}_x\to\mbA$. The space of such homomorphisms for $\mathbb{F}=\cF^*$ just consists of evaluations at points of $\cF$ and carries canonically a structure of a filtered bundle. Here we view $\mbA$ as a filtered (polynomial) associative algebra. In other words the duality comes from morphisms into $\mbA$, but $\mbA$ is for the passage $\cF\to\cF^*$ viewed as a pure filtered space, and for the passage $\cF^*\to\cF$ as a filtered algebra. The proof of the following is completely parallel to that for \cite[Theorem 2]{Grabowski:2016}.
\begin{theorem}
The associations $\cF\to\Hom_f(\cF,\mbA)$ and $\mathbb{F}\to\Hom_{f-alg}(\mathbb{F},\mbA)$  induce associations of the corresponding morphisms an establish an equivalence between the categories of filtered bundles and the category of filtered polynomial algebra bundles.
\end{theorem}

\subsection{Bachelor--Gaw\c{e}dzki like-theorems}
\begin{proposition}
There exists a canonical functor from the category of filtered bundles to the category of graded bundles which takes a filtered bundle $\cF$ of degree $k$ and produces a graded bundle $\Gr(\cF)$ of degree $k$, which isomorphic to $\cF$ as a filtered bundle.
\end{proposition}
\begin{proof}
We will prove the existence of such a functor by explicit construction. The idea is that from an atlas of filtered bundle of degree $k$ we an canonically associate an atlas of a graded bundle of degree $k$ by dropping the inhomogeneous pieces of the transition functions. On  $\mathcal{F}_{\{k\}}$ we employ natural local coordinates $(x^{a}, X^{I}_{w})$ together with the  fibre transformation law
$$X_{w}^{I'} = \sum_{w_{1} + \cdots + w_{n}\leq w} \frac{1}{n!} X^{J_{1}}_{w_{1}} X^{J_{2}}_{w_{2}} \cdots X^{J_{n}}_{w_{n}}T_{J_{n} \cdots J_{2} J_{1}}^{\hspace{30pt} I'}(x).$$
We then define the associated graded bundle, $F_{k}$ say,  by defining local coordinates $(x^{a}, y_{w}^{I})$ together with the fibre transformation law
\be\label{gr} y_{w}^{I'} = \sum_{w_{1} + \cdots + w_{n} = w} \frac{1}{n!}y^{J_{1}}_{w_{1}} y^{J_{2}}_{w_{2}} \cdots y^{J_{n}}_{w_{n}}T_{J_{n} \cdots J_{2} J_{1}}^{\hspace{30pt} I'}(x)\,.
\ee
The only thing that one has to verify is  that the cocycle condition remains true. However, this follows directly as the composition of graded affine transformations is again a graded affine transformation, and moreover any condition on such transformations can be examined order-by-order in weight (this is tightly related to the fact that we have a very particular tower of affine fibrations). Importantly, the lower order inhomogeneous terms in the transformation laws do not effect the cocycle condition for the homogeneous terms -- this can easily be seen as terms of a given weight cannot be sent to terms of a higher weight under graded affine transformations. This means that disregarding the inhomogeneous terms in the transformations laws does \emph{not} effect the cocycle condition for the remaining homogeneous terms. \par
In this way we build a graded bundle atlas for $F_{k}$, which is still a species of polynomial bundle. As a matter of formality, we will denote the association of a graded bundle with a filtered bundle as
$$\textnormal{Gr}(\mathcal{F}_{\{k\}}) = F_{k}.$$
Consider a morphism $\phi : \mathcal{F}_{\{k\}}  \rightarrow \mathcal{G}_{\{l\}} $  between filtered bundles (not necessarily of the same degree). We then define $\textnormal{Gr}(\phi)$ in terms of local coordinates by once again dropping the inhomogeneous pieces, see (\ref{eqn:filtered morphisms}). It is now a simple matter to see that the functorial properties are respected:
\begin{itemize}
\item Clearly $\textnormal{Gr}(\Id_{\mathcal{F}}) =  \Id_{\textnormal{Gr}(\mathcal{F})}$, and so identities are respected.
\item $\textnormal{Gr}(\psi \circ \phi) = \textnormal{Gr}(\psi) \circ \textnormal{Gr}(\phi) $, where $\psi: \mathcal{G} \rightarrow \mathcal{H}$, follows from the fact that in local coordinates morphisms are polynomials in the fibred coordinates and respect the degree.
\end{itemize}
Thus, we obtain a functor $\textnormal{Gr} : \catname{FilB} \rightarrow \catname{GrB}$, which can be restricted to the subcategories of filtered bundles and graded bundles  both of degree $k \in \mathbb{N}$.\par
 To show that $\Gr(\cF_{\{k\}})$ is isomorphic with $\cF_{\{k\}}$ as a filtered bundle, we will show that we can find local coordinates $(x^{a}, y_{w}^{I})$ on $\cF_{\{k\}}$ which transform as in (\ref{gr}).
We can do it passing to the graded algebra $\Gr(\A(\cF_{\{k\}}))$ associated canonically with the filtered algebra $\A(\cF_{\{k\}})$. Let us recall that the graded associative algebra associated with the filtration $\A=\cup_i\A_i$ is the graded space
$$\Gr(\A)=\A_0\oplus\bigoplus_i(\A_{i+1}/A_i)$$
with the obvious induced associative algebra structure. In our situation, the space $\A_{i+1}/A_i$ is the space of sections of a finite-dimensional vector bundle and the canonical map $\A\to\Gr(\A)$ is an isomorphism of filtered algebras. It is easy to see that $\Gr(\A(\cF_{\{k\}})$ is the polynomial algebra of functions on $\Gr(\cF_{\{k\}})$, so in every coordinate chart of $\cF_{\{k\}}$ we can choose homogeneous coordinates which, according to $\A\simeq\Gr(\A)$ correspond to global homogeneous functions on $\Gr(\cF_{\{k\}})$.
\end{proof}
\begin{remark}
The association $\cF_{\{k\}}\mapsto \Gr(\cF_{\{k\}})$ of a graded bundle with a filtered bundle is  unambiguous in the sense that the functor described above is canonical. What is non-canonical is the isomorphism $\cF_{\{k\}}\simeq \Gr(\cF_{\{k\}})$ which depends on the choice of `homogeneous coordinates'.
\end{remark}
It is important to note that the canonical functor $\textnormal{Gr}$  does \emph{not} give an equivalence of categories, simply  because it is not faithful. That is, as we drop parts of the initial transformation laws for the local coordinates on the filtered bundle, there is no way to recover them from the associated graded bundle. Namely, it is clear that we could have $\Gr(\phi)=\Gr(\psi)$ for morphisms $\phi,\psi:\cF_{\{k\}}^1\ra\cF_{\{k\}}^2$ which are different. It is enough that they have the same `highest order' terms.
\begin{example}
Consider an automorphism $\phi\ne Id$ of the filtered bundle $\R^{(1,1)}$ with coordinates $y,z$ of degree, respectively, 1 and 2:
$$\phi(y,z)=(y,z+y)\,.$$
It is clear that $\Gr(\phi)= Id$.
\end{example}
In fact the existence of a canonical functor provides a partial Bachelor--Gaw\c{e}dzki like-theorem. We have found coordinates on a graded bundle  associated with coordinates on a filtered bundle and a diffeomorphism between the two bundles --  the diffeomorphism in these special coordinates is simply the identity. That is, we can make the direct identification `$X =y$' and understand the filtered bundle and graded bundle to be the same as manifolds.  Thus we have the following theorem.
\begin{theorem}[Partial Bachelor--Gaw\c{e}dzki like-theorem]\label{thm:PBG}
Any filtered bundle  of degree $k$ is  isomorphic, but non-canonically, to a graded bundle of degree $k$
\end{theorem}
As any filtered bundle is isomorphic to a graded bundle and, in turn, any graded bundle is  isomorphic to a graded vector bundle (a Whitney sum of vector bundles), i.e., a \emph{split graded bundle} \cite{Bruce:2014}, we have the following theorem.
\begin{theorem}[Bachelor--Gaw\c{e}dzki like-theorem]\label{BG}
Any filtered bundle is non-canonically  isomorphic to a split graded bundle.
\end{theorem}

\begin{example}\label{exp:pplane}
Sometimes a filtered bundle is canonically partially split, that is we have a canonical graded bundle structure. For example, consider the manifold $\sJ^2_0(\R^p, M)$, elements of which are known as $(p,2)$-velocities. One can uncover the filtered structure here by considering local coordinates that are inherited from local coordinates on $M$.  Thus, it is natural to consider coordinates $(x^a, x^b_\mu , x^c_{\nu \lambda})$, where $\mu \in [1,p]$ and we take $\nu \leq \lambda $.  The admissible changes of local coordinates are of the form
\begin{align*}
& x^{a'} = x^{a'}(x), && x^{b'}_\mu = x^a_\mu \left(\frac{\partial x^{b'}}{\partial x^a} \right)\\
 &x^{c'}_{\nu \lambda} = x^b_{\nu \lambda}\left(\frac{\partial x^{c'}}{\partial x^b} \right)  + \frac{1}{2!}x^b_\nu x^d_\lambda \left (  \frac{\partial^2 x^{c'}}{\partial x^d \partial x^b} \right). &&
\end{align*}
Thus we can now treat $\mu \in [1,p]$ as the label \emph{defining} the weight. That is, we assign weight zero to $x^a$, $\mu$ to $x^b_\mu$ and $\nu+\lambda$ to $x^c_{\nu \lambda}$. With this assignment of weight we obtain a graded bundle of degree $2p$.  Almost identical considerations show that  $\sJ^k_0(\R^p, M)$ is a graded bundle of degree $kp$, and in particular $\sT^kM := \sJ^k_0(\R, M)$ is canonically a graded bundle of degree $k$.
\end{example}

\subsection{Towers of affine bundles vs filtered bundles}
\begin{definition}\label{def:affinetower}
A tower of fibre bundles of height $k$
\be\label{tower}B := B_k \stackrel{\tau_k}{\longrightarrow} B_{k-1} \stackrel{\tau_{k-1}}{\longrightarrow}  \cdots \stackrel{\tau_3}{\longrightarrow}   B_2 \stackrel{\tau_2}{\longrightarrow} B_{1 }\stackrel{\tau_1}{\longrightarrow} B_0 =: M\,,
\ee
is said to be an \emph{affine tower} of height $k$ if each level $B_i \stackrel{\tau_i}{\longrightarrow}B_j$ is an affine bundle.
\end{definition}
\begin{example}
An affine tower of height zero  is just a manifold $B_0 = :M$.
\end{example}
\begin{example}
An affine tower of height $1$, $\tau_1 : B_1 \longrightarrow M$ is just a `standard' affine bundle over $M$. Thus we make the identification $B_1 \cong \mathcal{F}_{\{1\}}$ via the obvious assignment of weight to the fibre coordinates.
\end{example}
\begin{example}
A filtered bundle of degree $k$ leads to an affine tower of height $k$,  see (\ref{eqn:tower}).
\end{example}
Clearly, a typical affine tower does not come from a filtered bundle.
\begin{definition}\label{def:filteredtower}
An affine tower of height $k$  is said to be a \emph{filterable tower} if it is isomorphic to a  tower of a filtered bundle, i.e., there exists a bundle isomorphism $\varphi: B_k  \rightarrow \mathcal{F}_{\{ k\}}$  (over  the same base $B_0 \cong F_{\{0\}} := M$)  commuting with all the projections $\tau_i$ in the towers.
\end{definition}
\noindent Note that we do not expect any such isomorphism to be unique nor in any way  canonical.\par

 It is completely clear that an affine tower is filterable if and only if it admits a \emph{filtered atlas}, i.e., an atlas with local coordinates constructed from coordinates in $M$ and affine coordinates $y_w$ in the fibers of $\zt_w:B_w\to B_{w-1}$, viewed as coordinates of degree $w$, such that the transition functions are weighted polynomials respecting the degree in the sense that coordinates of degree $\le i$ can be written as polynomials of weight $\le i$ with respect to the other set of coordinates.
\begin{example}
A general affine tower of height $2$, $ B_2 \longrightarrow B_1 \longrightarrow M$ admits adapted local  affine coordinates $(x^a, y^i , z^\alpha)$ and the admissible changes of coordinates are of the form
\begin{align*}
& x^{a'} = x^{a'}(x),&& y^{i'} = y^j T_j^{\:\: i'}(x) + T^{i'}(x),
&    z^{\alpha'} = y^\beta T_\beta^{\:\: \alpha'}(x,y) + T^{\alpha'}(x,y).
\end{align*}
Note that we have no further conditions on the coordinate transformations in general -- by comparison with  \eqref{eqn:tranformlaws} we see that these structures are more general than filtered bundles (also see \eqref{eqn:tower}). We can find a filtered atlas for the tower if we can choose coordinates so that
$T_\beta^{\:\: \alpha'}$ depends only on $x$ and $T^{\alpha'}(x,y)$ is of degree $\le 2$, i.e. a quadratic polynomial with respect to $y$.
\end{example}
The question of how to recognise a filterable tower,  which is an important step with identifying jet bundles, can be addressed \emph{inductively} level-by-level. At the first level there is nothing  much to do. We have a standard affine bundle $B_1 = \mathcal{F}_{\{ 1\}}$ and so by definition we are given an atlas consisting of charts adapted to the fibre bundle structure that consist of coordinates $\big(x^a , X^i_1 \big)$, such that the admissible changes of coordinates are of the form
\begin{align*}
& x^{a'} = x^{a'}(x), && X^{i'}_1 =  X^j_1T_j^{\:\: i'}(x) + T^{i'}(x).
\end{align*}
Thus, we can canonically assign a weight of zero to the base coordinates $x$ and one to the affine fibre coordinates $X_1$. The important observation is that within the space of functions on $B_1 = \mathcal{F}_{\{1\}}$, due to the existence of the privileged charts, there are the functions that are  weighted polynomials  with respect to the weight of the fibre coordinates (see Definition \ref{def:weightedpoly}). \par
Now consider the second level $B_{2} \rightarrow \mathcal{F}_{\{1\}}$, which by assumption is an affine bundle. Then $B_{2} \simeq \mathcal{F}_{\{ 2\}}$ if and only if there exists an atlas of $B_2$ consisting of charts adapted to the fibre bundle structure that consist of coordinates $\big(x^a, X^i_1 , X^l_2 \big)$, such that the admissible changes of coordinates are as above, together with
$$X^{l'}_2 =  X^m_2 T_m^{\:\: l'}\big (x, X_1 \big) + T^{l'}\big(x, X_1\big),$$
such that:
\begin{enumerate}
\item $T_m^{\:\: l'}\big (x, X_1 \big ) = T_m^{\:\: l'}(x)$, i.e. we have \emph{degree zero} weighted polynomials with respect to the weight of $X_1$;
\item  $T^{l'}\big(x, X_1  \big)$ are  \emph{degree two} weighted polynomials with respect to the weight of $X_1$.
\end{enumerate}
We thus naturally assign a weight of two to the coordinates $X_2$ and obtain a filtered bundle.\par
The situation for the third level is clear. The essential step is that we now have a well defined notion of weighted polynomials on $B_2 \simeq \mathcal{F}_{\{2 \}}$ and so can make sense of weighted polynomials of any integer degree, specifically zero and three. That is,  on the affine bundle $B_3 \rightarrow \mathcal{F}_{\{2 \}}$, we require the existence of local coordinates such that the non-linear term in the affine changes of coordinates are  weighted polynomials of degree $3$.  Then naturally we can assign a weight of $3$ to the respective fibre coordinates,  and so we can make the identification $B_3 \simeq \mathcal{F}_{\{ 3\}}$.  And then we continue this process until we reach the top level. This then reduces the question to the study of an affine bundle over a filtered bundle and the existence of a `reduced' or `refined' atlas compatible with the lower filtered structure.
\begin{theorem}\label{T1} A tower of affine fibrations (\ref{tower}) is filterable if and only if, for $i=2,\dots, k$, the model vector bundle
$\sv(\tau_i):\sv(B_i)\to B_{i-1}$ of the affine fibration $\tau_i:B_i\to B_{i-1}$ is the pull-back bundle
$(\tau^{i-1}_0)^*V_i$ of a certain vector bundle $\zs_i:V_i\to M$ for all $i=2,\dots, k$. Here, $\tau^{i-1}_0:B_{i-1}\to B_0=M$ is the canonical projection of $B_{i-1}$ onto $M$.
\end{theorem}
\begin{proof}
This is a necessary condition according to Theorem \ref{T0}. As for sufficiency we can proceed
inductively, choosing a section of the affine fibration $\tau_i:B_i\to B_{i-1}$ (which always exists as the fibers are contractible), we can identify $B_i$ with $\sv(B_i)$. The transition functions in fibers of $\sv(B_i)\to B_{i-1}$ are the linear parts of the affine transition functions in fibers of $\tau_i:B_i\to B_{i-1}$, so the tower is filterable if and only if we can choose the coefficients of these linear transformations depending only on $x\in M$ (they generally depend on $y\in B_{i-1}$). This mean exactly that the identification of affine fibers determined by the coordinate system survives along fibers of the projection $B_{i-1}\to M$ under the changes of coordinates. In other words, the vector bundle $\sv(B_i)\to B_{i-1}$ must be the pull-back along the projection $B_{i-1}\to M$ of a vector bundle $V_i$ over $M$.
\end{proof}
\begin{theorem} \label{T2} A filterable tower of affine fibrations determines the filtered bundle up to an isomorphism.
\end{theorem}
\begin{proof} According to our version of Bachelor--Gaw\c{e}dzki theorem \ref{BG}, any filtered bundle associated with a filterable tower of affine fibrations is isomorphic with a graded vector bundle $V_1\oplus_M\cdots\oplus_MV_k$ for some vector bundles $V_i$ over $M$. The proof of Theorem \ref{T1} shows that $V_i$ satisfy the condition
$$\sv(B_i)=(\tau^{i-1}_0)^*V_i\,.$$
But this determines $V_i$ up to isomorphisms if only $B_i$ is given.
\end{proof}
\begin{example}
It is well know that for  smooth maps between two manifolds $M$ and $N$ that  $$\tau_{k-1}^k : \sJ^k (M,N)  \longrightarrow \sJ^{k-1}(M,N) $$
is an affine bundle whose model vector bundle is the pullback of $\sT N \otimes \textnormal{S}^k \sT^*M$ over $J^{k-1}(M,N)$  (see  \cite[Proposition 12.11.]{Kolar:1993}). Thus, the tower of affine bundles
$$\sJ^k(M,N) \stackrel{\tau^{k}_{k-1}}{\xrightarrow{\hspace*{20pt}} } \sJ^{k -1}(M,N)  \stackrel{\tau^{k-1}_{k-2}}{\xrightarrow{\hspace*{20pt}} } \sJ^{k -2}(M,N)   \longrightarrow  \cdots \longrightarrow \sJ^1(M,N)  \stackrel{\tau^{1}_{0}}{\xrightarrow{\hspace*{15pt}} } \sJ^0(M,N) := N \times M,$$
is, via Theorem \ref{T1}, filterable. In cosequence, via Theorem \ref{T2}, $ \sJ^k (M,N)$ is a filtered bundle. From Example \ref{exp:jetbundles} taking a trivial bundle we see that $\sJ^1(M,N) \simeq \sT N \otimes \sT^* M$ and so we canonically have a vector bundle at the lowest level (this can of course be deduced in several different ways).
 \end{example}
\begin{coexample}
Let $M$ be a $m+n$ dimensional smooth manifold. Then for any $n$-dimensional submanifold $N \subset M$ we can define $[N]_x^k$ as the $k$-th jet of $N$ at $x\in M$.  That is, we consider the set of $n$-dimensional submanifolds of $M$ that are in contact to order $k$ at $x$. This can be defined using local parametrisations $\psi : \R^n \rightarrow M $ of $N$ and $\psi^\prime :\R^n \rightarrow M$ of $N^\prime$ such that $\psi(0) = \psi^\prime(0)$, and then considering the jets of these maps at zero (see for example \cite[page 124]{Kolar:1993}).  We then define the \emph{k-jet of submanifolds of $M$} as
$$\sJ_n^k(M) := \bigcup_{x \in M} \sJ^k_n(M)_x, $$
where $\sJ^k_n(M)_x := \left \{ [N]_x ^k ~|~ x \in M   \right \}.$  It is well known that there is a series of bundle structures here given by reducing the order of contact
$$\tau_m^l  : \sJ_n^l(M) \longrightarrow \sJ_n^m(M),$$
for $l < m \leq k$.  Here we  define $\sJ_n^0(M) := M$ as standard.  In particular, we have affine fibre bundle structures
$$\tau_{l-1}^l  : \sJ_n^l(M) \longrightarrow \sJ_n^{l-1}(M),$$
 provided $l>1$. The  $l=1$ case needs special attention,
$$\tau^1_0 : \sJ^1_n(M) \longrightarrow M.$$
It is not too hard to see that the fibration $\tau^1_0$ is isomorphic to the \emph{Grassmann bundle} of $n$-dimensional subspaces of $\sT M$. In particular, there is no affine structure of the fibration $\tau^1_0$.  Thus, the tower of fibre bundles
$$\sJ_n^k(M) \stackrel{\tau^{k}_{k-1}}{\xrightarrow{\hspace*{20pt}} } \sJ^{k -1}_n(M)  \stackrel{\tau^{k-1}_{k-2}}{\xrightarrow{\hspace*{20pt}} } \sJ^{k -2}_n(M)   \longrightarrow  \cdots \longrightarrow \sJ^1_n(M)  \stackrel{\tau^{1}_{0}}{\xrightarrow{\hspace*{15pt}} } \sJ^0_n(M) := M.$$
is \emph{not} a tower of affine fibrations and so $\sJ^k_n(M)$ is not canonically a filtered bundle.
\end{coexample}

\section{Multiple structures and linearisation}\label{sec:MultipleFiltered}
\subsection{Filtered-linear bundles}
\begin{definition}
A \emph{filtered-linear bundle} is a vector bundle in the category of filtered bundles.
\end{definition}
\begin{remark}
Via `categorical nonsense' we know that we can reverse the definition and consider a filtered linear bundle as a filtered bundle in the category of vector bundles.
\end{remark}
The most economical way to define a vector bundle is in terms of a regular action of the multiplicative monoid of reals $(\mathbb{R}, \cdot)$ on a manifold (see \eqref{eqn:HomoReg} and \cite{Grabowski:2009,Grabowski:2012}). In this way we understand a filtered-linear bundle to be a filtered bundle together with regular action of $(\mathbb{R}, \cdot)$  that acts as a morphism of filtered bundles for all `time' $t \in \mathbb{R}$. This is quite consistent with the understanding of double vector bundles given in \cite{Grabowski:2009}. \par
Thus, on any filtered-linear bundle we can find natural coordinates adapted to the structure
$$(\underbrace{x^{a}}_{(0,0)},~ \underbrace{X^{I}_{w}}_{(w,0)} ; ~ \underbrace{Y^{\Sigma}_{u}}_{(u,1)}),$$
$(0\leq u \leq k)$ where we have denote the filtered degree and the natural   weight as coordinates on a vector bundle, together with the transformation law for the linear coordinates
$$Y^{\Sigma'}_{u} = \sum_{u_{0} + w_{1} + \cdots w_{n} \leq u} \frac{1}{n!} Y^{\Upsilon}_{u_{0}} X^{J_{1}}_{w_{1}} X^{J_{2}}_{w_{2}} \cdots X^{J_{n}}_{w_{n}}T_{ J_{n} \cdots  J_{1} \Upsilon}^{\hspace{30pt} \Sigma'}(x). $$
We take the coefficients to be symmetric in the $J$'s.
\begin{remark}
Filtered-graded bundles can similarly be defined as a filtered bundle equipped with a smooth action of the multiplicative monoid of reals that acts as morphisms of filtered bundles for all time, that is we drop the regularity condition.
\end{remark}
Given an arbitrary filtered bundle $\mathcal{F}_{\{k\}}$ of degree $k$ its tangent bundle $\sT \mathcal{F}_{\{k\}}$  naturally inherits natural  coordinates
$$(\underbrace{x^{a}}_{(0,0)} ,~ \underbrace{X^{I}_{w}}_{(w,0)}; ~\underbrace{\delta x^{b}}_{(0,1)}, ~ \underbrace{\delta X^{I}_{w}}_{(w,1)} ),$$
 The transformation law of the fibre coordinates are
\begin{eqnarray}
\label{eqn:TangentTrans}
\delta x^{a'} &=& \delta x^{b} \frac{\partial x^{a'}}{\partial x^{b}}, \\
\nonumber
\delta X^{I'}_{w} &=& \sum_{w_{1} + \cdots + w_{n}\leq w} \frac{1}{(n-1)!} \delta X^{J_{1}}_{w_{1}} X^{J_{2}}_{w_{2}} \cdots X^{J_{n}}_{w_{n}}T_{J_{n} \cdots J_{2} J_{1}}^{\hspace{30pt} I'}\\
\nonumber
 &+& \sum_{w_{1} + \cdots + w_{n}\leq w} \frac{1}{n!} X^{J_{1}}_{w_{1}} X^{J_{2}}_{w_{2}} \cdots X^{J_{n}}_{w_{n}}  \delta x^{a} \frac{\partial }{\partial x^{a}}T_{J_{n} \cdots J_{2} J_{1}}^{\hspace{30pt} I'}.
\end{eqnarray}
We will refer to the inherited filtered structure on the tangent bundle as the \emph{tangent lift} of the filtered bundle. A simple observation based on the above transformation laws is the following.
\begin{proposition}\label{prop:Ttower}
Given an arbitrary filtered bundle, $\mathcal{F}_{\{k\}}$ there exists a tower of fibre bundle structures
\begin{center}
\leavevmode
\begin{xy}
(0,20)*+{\sT \mathcal{F}_{\{ k\}}}="a"; (20,20)*+{\sT \mathcal{F}_{\{ k-1\}}}="b";  (40,20)*+{\cdots }="c"; (60,20)*+{\sT \mathcal{F}_{\{ 1\}}}="d"; (80,20)*+{\sT M}="e";   %
 {\ar "a";"b"}; {\ar "b";"c"}; {\ar "c";"d"}; {\ar "d";"e"};%
(0,0)*+{ \mathcal{F}_{\{ k\}}}="f"; (20,0)*+{\mathcal{F}_{\{ k-1\}}}="g";  (40,0)*+{\cdots }="h"; (60,0)*+{ \mathcal{F}_{\{ 1\}}}="i"; (80,0)*+{ M}="j";   %
{\ar "f";"g"}; {\ar "g";"h"}; {\ar "h";"i"}; {\ar "i";"j"};%
{\ar "a";"f"}; {\ar "b";"g"}; {\ar "d";"i"};  {\ar "e";"j"}%
\end{xy}
\end{center}
where the horizontal arrows are filtered bundle structures and the vertical arrows are vector bundle structures.
\end{proposition}
Another simple observation based on the transformation laws (\ref{eqn:TangentTrans})  is the following.
\begin{proposition}
The tangent functor $\sT$ and the canonical functor $\textnormal{Gr}$ commute.
\end{proposition}
In simpler terms, we end up with the same filtered-linear bundle independently of the  order in which we apply the tangent and canonical functor (up to natural isomorphisms).\par
The notion of the vertical bundle $\sV \mathcal{F}_{\{k\}}$ as a substructure of the tangent bundle  is clear, and this to is a filtered-linear bundle. Natural coordinates on the vertical bundle are $(x^{a},~ X^{I}_{w},~ \delta X^{I}_{w})$, the transformation laws for which are given by those of the tangent bundle by setting $\delta x =0$. A simple observation here is the following.
\begin{proposition}\label{prop:Vtower}
Given an arbitrary filtered bundle, $\mathcal{F}_{\{k\}}$ there exists a tower of fibre bundle structures
\begin{center}
\leavevmode
\begin{xy}
(0,20)*+{\sV \mathcal{F}_{\{ k\}}}="a"; (20,20)*+{\sV \mathcal{F}_{\{ k-1\}}}="b";  (40,20)*+{\cdots }="c"; (60,20)*+{\sV \mathcal{F}_{\{ 1\}}}="d";  %
 {\ar "a";"b"}; {\ar "b";"c"}; {\ar "c";"d"};%
(0,0)*+{ \mathcal{F}_{\{ k\}}}="f"; (20,0)*+{\mathcal{F}_{\{ k-1\}}}="g";  (40,0)*+{\cdots }="h"; (60,0)*+{ \mathcal{F}_{\{ 1\}}}="i";;   %
{\ar "f";"g"}; {\ar "g";"h"}; {\ar "h";"i"}; %
{\ar "a";"f"}; {\ar "b";"g"}; {\ar "d";"i"};  %
\end{xy}
\end{center}
where the horizontal arrows are filtered bundle structures and the vertical arrows are vector bundle structures.
\end{proposition}
Passing to the cotangent bundle we encounter a complication. As the `momentum' associated with a given coordinate transform as derivatives, it is clear that the naturally inherited fibre coordinates on the cotangent bundle  of a filtered  bundle have negative weight. Thus we leave the category of filtered bundles. The solution to this problem is to consider the \emph{phase lift} as first defined for graded bundles in \cite{Grabowski:2013}. Essentially the phase lift is a shift in the weight of the momenta by $+k$, assuming we have a filtered bundle of degree $k$. We will always use this phase shift when dealing with cotangent bundles of filtered and graded bundles. Thus on $\sT^{*}\mathcal{F}_{\{k\}}$ we can employ natural coordinates
$$(\underbrace{x^{a}}_{(0,0)}, ~ \underbrace{X^{I}_{w}}_{(w,0)}; ~ \underbrace{\pi_{a}^{k}}_{(k,1)},~ \underbrace{P_{J}^{k-w}}_{(k-w,1)}).$$
The transformation laws for the momenta can easily be derived and the readers can quickly convince themselves that we do indeed not leave the category of filtered bundles (upon ignoring the vector bundle structure).\par
There are several interesting substructures of the cotangent bundle of a filtered bundle. Of potential interest for Hamiltonian mechanics is the Mironian (see \cite{Miron:2003} for the classical case)
$$\tau : \sT^{*}\mathcal{F}_{\{k\}} \rightarrow  \textnormal{Mi}(\mathcal{F}_{\{k\}}),$$
where we (locally) define the Mironian using local coordinates
$$(x^{a}, Y^{\Sigma}_{u}, P_{j}^{0})$$
where we have defined $X^{I}_{w} := (Y^{\Sigma}_{u}, Z_{k}^{j})$, with $0 < u <k$. From the transformation laws it is clear that
$$\textnormal{Mi}(\mathcal{F}_{\{k\}}) \simeq \mathcal{F}_{\{k\}} \times_{M} \sv\big(\mathcal{F}_{\{k\}}\big)^*.$$
\begin{remark}
While there is an analogous tower of fibre bundle structures associated the cotangent bundle of a filtered bundle, and indeed any filtered-linear bundle, the tower is not simply given by that of proposition (\ref{prop:Ttower}) (or proposition (\ref{prop:Vtower})) upon the replacement $\sT \mapsto \sT^{*}$.
\end{remark}
Another example of a filtered-linear bundle is provided by the dual of the vertical bundle $\sV^{*}\mathcal{F}_{\{ k\}}$, which is not a substructure of $\sT^{*}\mathcal{F}_{\{k\}}$, but rather a quotient of it.  Natural local coordinates on the dual vertical bundle are
$$(\underbrace{x^{a}}_{(0,0)}, ~ \underbrace{X^{I}_{w}}_{(w,0)}, ~ \underbrace{\Psi_{J}^{k-w+1}}_{(k-w+1,1)}),$$
and the admissible changes of linear coordinates can be deduced from the invariant pairing
$$\delta X_{w}^{I}\Psi_{I}^{k-w+1}.$$
The assignment of the bi-weight is chosen, like in the case of the cotangent bundle, in order for us not to leave the category of filtered bundles.

\begin{example}
Let us consider $\sV^{*}\mathcal{F}_{\{2\}}$, which we endow with local coordinates
$$(\underbrace{x^{a}}_{(0,0)},~ \underbrace{Y^{\alpha}}_{(1,0)}, ~ \underbrace{Z^{i}}_{(2,0)}; ~ \underbrace{\psi_{\beta}}_{(2,1)},~ \underbrace{\chi_{i}}_{(1,1)}).$$
The transformation laws for the linear coordinates can be deduced from the  invariant pairing
$$\delta Y^{\alpha}\psi_{\alpha} + \delta Z^{i}\chi_{i},$$
and after some straightforward algebra we arrive at
\begin{eqnarray*}
\psi_{\beta'} &=& T_{\beta'}^{\:\:\: \alpha} \psi_{\alpha} - \left(  Y^{\gamma}T_{\beta'}^{\:\:\:\alpha}T_{\alpha \gamma}^{\:\:\:\:\:\: j'}T_{j'}^{\:\:\: l} + T_{\beta'}^{\:\:\: \alpha} T_{\alpha}^{\:\:\: j'}T_{j'}^{\:\:\: l} \right)\chi_{l},\\
\chi_{j'} &=& T_{j'}^{\:\:\: i}\chi_{i}.
\end{eqnarray*}
Clearly these changes of coordinates are consistent with out assignment of the bi-weight and the expected filtration.
\end{example}
From the definition of a filtered-linear bundle, the notion of a morphism between filtered-linear bundles is clear: we have smooth morphisms as polynomial bundles that respect both the additional structures.  Evidently, under the standard composition of smooth maps we obtain the category of filtered-linear bundles, which we denote as $\catname{FilLB}$. Clearly we have the obvious forgetfull functor $\catname{FilLB} \rightarrow \catname{FilB}$ given by forgetting the linear structure. If we restrict attention to filtered-linear bundles where the filtered structure  is of degree $k$, then we obtain the full subcategory of  filtered-linear bundles of degree $k$, which we denote as $\catname{FilLB}[k]$.  Again, we have the obvious forgetful functor $\catname{FilLB}[k] \rightarrow \catname{FilB}[k]$.  
 \subsection{Linearisation}
One of the most useful features of the multiple graded bundles is the possibility of playing with weight vector fields. Since linear combination of weight vector fields with positive integer coefficients is again a weight vector field, we can use these new weights to analyze the structure of the bundle. We can also use negative integer coefficients if all the resulting weights are positive. This mechanism was used in \cite{Bruce:2014} to define linearisation of a graded bundle, and in \cite{Bruce:2016} iteration of the linearisation was used to define the full linearisation of a graded bundle. In case of weighted linear bundles we do not have weight vector fields for both structures, but the nature of transformation laws still makes it possible to shift weights in a filtered part of the bundle. Indeed let us use the coordinates
$$(\underbrace{x^{a}}_{(0,0)},~ \underbrace{X^{I}_{w}}_{(w,0)} ; ~ \underbrace{Y^{\Sigma}_{u}}_{(u,1)}),$$
adapted to the structure of a filtered-linear bundle.  The admissible transformations of coordinates are of the form
\begin{eqnarray*}
x^{a'} &=& x^{a'}(x), \\
X_{w}^{I'} &=&  \sum_{w_{1} + \cdots + w_{n}\leq w} \frac{1}{n!} X^{J_{1}}_{w_{1}} X^{J_{2}}_{w_{2}} \cdots X^{J_{n}}_{w_{n}}T_{J_{n} \cdots J_{2} J_{1}}^{\hspace{30pt} I'}(x), \\
Y^{\Sigma'}_{u} &=& \sum_{u_{0} + w_{1} + \cdots w_{n} \leq u} \frac{1}{n!} Y^{\Upsilon}_{u_{0}} X^{J_{1}}_{w_{1}} X^{J_{2}}_{w_{2}} \cdots X^{J_{n}}_{w_{n}}T_{ J_{n} \cdots  J_{1} \Upsilon}^{\hspace{30pt} \Sigma'}(x).
\end{eqnarray*}
We can now shift the ``filtered weight'' by adding the ``linear weight'', i.e. we have
$$(\underbrace{x^{a}}_{(0,0)},~ \underbrace{X^{I}_{w}}_{(w,0)} ; ~ \underbrace{Y^{\Sigma}_{u+1}}_{(u+1,1)}).$$
The first two rows of the above coordinate transformation equations remain unchanged. In the transformation rules for linear coordinates $Y^I_u$, a linear coordinate has to appear exactly once in every monomial, therefore the total weight of every monomial increases by one, so does the weight of the new linear coordinate. The transformation rules are therefore consistent with the new filtered weight assignment. Let us note that we can also decrease the filtered weight of the linear coordinates if the initial filtered weight of all the linear coordinates is strictly positive. This is the case of a filtered linear bundle $\sV \mathcal{F}$ for any filtered bundle $\mathcal{F}$. Let us now distinguish the coordinates with the highest weight  $k$
$$(x^a, Y^I_w, Z^\alpha_k)\,,\qquad 1\leq w< k.$$
On $\sV \mathcal{F}$ we have natural coordinates
$$(\underbrace{x^{a}}_{(0,0)},\, \underbrace{Y^I_w}_{(w,0)},\, \underbrace{Z^\alpha_k}_{(k,0)};\, \underbrace{\delta Y^i_w}_{(w,1)},\, \underbrace{\delta Z^\alpha_k)}_{(k,1)}$$
Decreasing the filtered weight by linear weight we get
$$(\underbrace{x^{a}}_{(0,0)},\, \underbrace{Y^I_w}_{(w,0)},\, \underbrace{Z^\alpha_k}_{(k,0)};\, \underbrace{\delta Y^i_{w-1}}_{(w-1,1)},\, \underbrace{\delta Z^\alpha_{k-1}}_{(k-1,1)})$$
\begin{definition}\label{def:linearisation}
The \emph{linearisation} of the filtered bundle $\mathcal{F}$ of degree $k$ is a filtered-linear bundle $\pLinr(\mathcal{F})$ which is the second term $(\sV\mathcal{F})_{\{k-1\}}$ in the tower of fibrations (\ref{eqn:tower}) for filtered structure of the filtered-linear bundle $\sV\mathcal{F}$ with shifted filtered weight.
\end{definition}
Natural coordinates in $\pLinr(\mathcal{F})$ are then
$$(\underbrace{x^{a}}_{(0,0)},\, \underbrace{Y^I_w}_{(w,0)};\, \underbrace{\delta Y^i_{w-1}}_{(w-1,1)},\, \underbrace{\delta Z^\alpha_{k-1}}_{(k-1,1)})$$
with admissible transformations of the form
\begin{eqnarray*}
x^{a'} &=& x^{a'}(x), \\
Y_{w}^{I'} &=&  \sum_{w_{1} + \cdots + w_{n}\leq w} \frac{1}{n!} Y^{J_{1}}_{w_{1}} Y^{J_{2}}_{w_{2}} \cdots Y^{J_{n}}_{w_{n}}T_{J_{n} \cdots J_{2} J_{1}}^{\hspace{30pt} I'}(x), \\
\delta Y^{I'}_{w} &=& \sum_{1\leq w_{1} + \cdots w_{n} \leq w} \frac{1}{(n-1)!} \delta Y^{J_{1}}_{w_{1}} Y^{J_{2}}_{w_{2}} \cdots Y^{J_{n}}_{w_{n}}T_{ J_{n} \cdots  J_{1}}^{\hspace{30pt} I'}(x) ,\\
\delta Z^{\alpha'}_{k} &=& \delta Z^\beta S^{\alpha'}_\beta(x)+\sum_{1\leq w_{1} + \cdots w_{n} \leq k} \frac{1}{(n-1)!} \delta Y^{J_{1}}_{w_{1}} Y^{J_{2}}_{w_{2}} \cdots Y^{J_{n}}_{w_{n}}R_{ J_{n} \cdots  J_{1}}^{\hspace{30pt} I'}(x).
\end{eqnarray*}
which means that $\pLinr(\mathcal{F})$ is a filtered-linear bundle with filtered order $k-1$.  By convention, if the actual (minimal) degree of $\mathcal{F}$ is $<k$, then the above procedure of linearisation in degree $k$ stops at $\sV\mathcal{F}$, as there are no coordinates of degree $k$ to be removed.

It is easy to see that the tower of fibrations associated to the filtered structure of $\pLinr(\mathcal{F})$ is
$$\pLinr(\mathcal{F})\longrightarrow \pLinr(\mathcal{F}_{\{k-1\}})\longrightarrow\cdots\rightarrow \pLinr(\mathcal{F}_{\{2\}})\longrightarrow \pLinr(\mathcal{F}_{\{1\}})\,.$$ Note that the filtered degree of $\pLinr(\mathcal{F}_{\{i\}})$ is $i-1$. \par
\begin{remark}
Heuristically, one should view the  linearisation of a  filtered bundle as a partial polarisation of the admissible changes of local coordinates. That is,  we adjoin new coordinates in a natural way as to linearise part of the polynomial changes of fibre coordinates.  The new coordinates essentially come from  applying the tangent functor and then reducing the resulting structure.
\end{remark}
\begin{example}\label{exp:linearisation affine} Let $\pi: \mathsf{A}\rightarrow M$  be an affine bundle, i.e., filtered bundle of degree $1$. The model vector bundle will be denoted by $\sv(\mathsf{A})$. There is the canonical isomorphism $\sV \mathsf{A}\simeq \mathsf{A}\times_M\sv(\mathsf{A})$. If coordinates in $\mathsf{A}$ are $(x^a, X^I_1)$, then coordinates on $\sV \mathsf{A}$ are $(x^a, X^I_1, \delta X^I_1)$, where coordinates $\delta X^I_1$ carry double filtered-linear weight $(1,1)$. After weight shifting the only coordinates with filtered weight equal to $1$ are $X^I_1$ in fibers of the affine bundle $\mathsf{A}$. Projection $\mathsf{A}\times_M\sv(\mathsf{A})\rightarrow\pLinr(\mathsf{A})$ coincides with the projection
on the second factor, i.e., $\pLinr(\mathsf{A})=\sv(\mathsf{A})$.  In other words, the linearisation of an affine bundle is its vector bundle model.
\end{example}
\begin{remark}
Note that Example \ref{exp:linearisation affine} shows that, on the contrary to graded case, there is no canonical embedding of the filtered bundle $\mathcal{F}$ into its linearisation.
\end{remark}
From the definition/construction of the linearisation and Example \ref{exp:linearisation affine} we arrive at the following observation.
\begin{proposition}
Canonically associated with the linearisation of any filtered bundle of degree $k$  is the following ladder of fibre bundle structures,
$$\xymatrix{
\pLinr(\mathcal{F}_{\{k\}})\ar[d] \ar[r] & \pLinr(\mathcal{F}_{\{k-1\}})\ar[d] \ar[r] & \cdots \ar[r] & \pLinr(\mathcal{F}_{\{2\}})\ar[d]\ar[r] & \pLinr(\mathcal{F}_{\{1\}}) = \sv(\mathcal{F}_{\{1\}})\ar[d] \\
\mathcal{F}_{\{k-1\}}\ar[r] & \mathcal{F}_{\{k-2\}}\ar[r] & \cdots \ar[r] & \mathcal{F}_{\{1\}}\ar[r] & M =: \mathcal{F}_{\{ 0 \}}
}$$
where the vertical arrows are vector bundle structures, and the   horizontal arrows  are  general filtered bundle structures.
\end{proposition}
\begin{example}
Continuing Example \ref{exp:jetbundles}, we examine the linearisation of the jet bundle $\sJ^2 E$, where $\pi : E \rightarrow M$ is fibre bundle. Recall that $\sJ^2 E$ can be equipped with adapted local coordinates $(x^{a},~ y^{\alpha},~z^{\beta}_{b}, ~ w^\gamma_{cd})$, where as a filtered bundle the coordinates $(x, y)$ on $E$ are of weight zero, while $z$ and $w$ are of weight $1$ and $2$, respectively. Following the `recipe' for the construction of the linearisation, we see that $\pLinr(\sJ^2 E)$ comes with naturally induced coordinates
$$\big ( \underbrace{x^a}_{(0,0)}, \: \underbrace{y^\alpha}_{(0,0)} , \: \underbrace{z_b^\beta}_{(1,0)} , \: \underbrace{\delta z_c^\delta}_{(0,1)}, \: \underbrace{\delta w_{cd}^\epsilon}_{(1,1)}\big). $$
The admissible changes of the `new' coordinates can be deduced directly by taking derivatives and shown to be
\begin{align*}
 & \delta z_{a'}^{\alpha'} = \left(\frac{\partial x^b}{\partial x^{a'}} \right)\left(\delta z_b^\beta \frac{\partial }{\partial y^\beta} \right) y^{\alpha'}(x,y),\\
 & \delta w_{a'b'}^{\alpha'} = \left( \frac{\partial x^c}{\partial x^{a'}}\right)\left(\delta z_c^\beta  \frac{\partial}{\partial y^\beta} + \delta w_{cd}^\beta  \frac{\partial}{\partial  z_d^\beta}\right) z_{b'}^{\alpha'} (x,y,z)\\
  & + \left( \frac{\partial x^c}{\partial x^{a'}}\right) \delta z_e^\gamma   \left(  \left(\frac{\partial x^e}{\partial x^{b'}} \right) \frac{\partial^2}{\partial x^c \partial y^\gamma}  {-}  \left(\frac{\partial x^d}{\partial x^{b'}} \right)\frac{\partial^2 x^{f'}}{\partial x^d \partial x^c} \left( \frac{\partial x^e}{\partial x^{f'}}\right)\frac{\partial}{\partial y^\gamma}\right)y^{\alpha'}(x,y).
\end{align*}
We have the following commutative diagram of fibrations:
\begin{center}
\leavevmode
\begin{xy}
(0,20)*+{\pLinr(\sJ^2E)}="a"; (20,20)*+{ \sv(\sJ^1E)}="b";  %
(0,0)*+{ \sJ^1E}="c"; (20,0)*+{E}="d";  %
{\ar "a";"b"}; {\ar "c";"d"};%
{\ar "a";"c"}; {\ar "b";"d"};%
\end{xy}\,,
\end{center}
where the vertical arrow are vector bundle structures, while the horizontal arrows are  general filtered bundle structures. Note that $\pLinr(\sJ^1E) = \sv(\sJ^1E)$ from Example \ref{exp:linearisation affine}.
\end{example}
\begin{remark}
As far as we know, the linearisation of jet bundles $\sJ^kE$, other than the first jet bundle where we have  $\pLinr(\sJ^1E) = \sv(\sJ^1E)$, has not appeared in the previous literature. It is certainly our opinion that this construction deserves further study.
\end{remark}

Similarly to the case of the linearisation of graded bundles we have the following theorem
\begin{theorem}\label{thm:LinearisationFunctor}
The linearisation  of a filtered bundle can be considered as a functor
$$\pLinr : \catname{FilB}[k] \rightarrow \catname{FilLB}[k-1],$$
 from the category of filtered bundles of degree $k$ to the category of filtered-linear bundles of degree $k-1$.
\end{theorem} 
\begin{proof}
The action of  linearisation on an object in $\catname{FilB}[k]$ --  by construction  -- produces an object in $\catname{FilLB}[k-1]$. The only thing we have to prove that for every morphism $\phi:\mathcal{F}_k\rightarrow\mathcal{F}'_k$ there exists a unique  morphism of the filtered-linear bundles $\pLinr(\phi) : \pLinr(\mathcal{F}_k)\rightarrow \pLinr(\mathcal{F}'_k)$. We do this via explicit construction. It is clear that we can apply the vertical tangent functor $\sV$ to the morphism $\phi$. What we get is a morphism of filtered-linear bundles $\sV\mathcal{F}_k$ and $\sV\mathcal{F}'_k$. Indeed using adapted coordinates with distinguished highest weight $(x^a, X^I_w, Z^\alpha_k)$ in $\mathcal{F}_k$ and $(\bar x^i, \bar X^P_u, \bar Z^\mu_k)$ in $\mathcal{F}'_k$ we can express $\phi$ in the following form
\begin{eqnarray*}
\bar x^i\circ\phi(x)&=&\phi^i(x), \\
\bar X^P_u\circ\phi(x,X)&=&\sum_{w_{1} + \cdots + w_{n}\leq u} \frac{1}{n!} X^{I_{1}}_{w_{1}} X^{I_{2}}_{w_{2}} \cdots X^{I_{n}}_{w_{n}}\Phi_{I_{n} \cdots I_{2} I_{1}}^{\hspace{30pt} P}(x),\\
\bar Z^\mu_k\circ\phi(x,X,Z)&=&Z^\alpha_k\Psi^\mu_\alpha(x) + \sum_{w_{1} + \cdots + w_{n}\leq k} \frac{1}{n!} X^{I_{1}}_{w_{1}} X^{I_{2}}_{w_{2}} \cdots X^{I_{n}}_{w_{n}}\Xi_{I_{n} \cdots I_{2} I_{1}}^{\hspace{30pt} \mu}(x).
\end{eqnarray*}
Applying $\sV$ we get the expressions for linear coordinates
\begin{eqnarray*}
\delta \bar X^P_u\circ\sV\phi(x,X,\delta X)&=&\sum_{w_{1} + \cdots + w_{n}\leq u} \frac{1}{(n-1)!} \delta X^{I_{1}}_{w_{1}} X^{I_{2}}_{w_{2}} \cdots X^{I_{n}}_{w_{n}}\Phi_{I_{n} \cdots I_{2} I_{1}}^{\hspace{30pt} P}(x),\\
\delta \bar Z^\mu_k\circ\sV\phi(x,X,Z, \delta X, \delta Z)&=&\delta Z^\alpha_k\Psi^\mu_\alpha(x) + \sum_{w_{1} + \cdots + w_{n}\leq k} \frac{1}{(n-1)!} \delta X^{I_{1}}_{w_{1}} X^{I_{2}}_{w_{2}} \cdots X^{I_{n}}_{w_{n}}\Xi_{I_{n} \cdots I_{2} I_{1}}^{\hspace{30pt} \mu}(x).
\end{eqnarray*}
The only place where there appear coordinates $Z^\alpha_k$ of highest weight is in the expression for $\bar Z^\mu_k\circ\phi(x,X,Z)$, so we can drop coordinates $Z^\alpha_k$ and
$\bar Z^\mu_k$ simultaneously obtaining an expression for a map from $\pLinr(\mathcal{F}_k)$ to $\pLinr(\mathcal{F}'_k)$. This map we define to be $\pLinr(\phi)$. From the coordinate expression we see that $\pLinr(\phi)$ is indeed a morphism of filtered-linear bundles. It is easy to see that $\pLinr(\mathrm{id}_{\mathcal{F}})=\mathrm{id}_{\pLinr({\mathcal{F}})}$ and $\pLinr(\phi\circ\psi)=\pLinr(\phi)\circ\pLinr(\psi)$. The latter follows from the appropriate property of the vertical tangent functor. Thus the linearisation satisfies all the necessary properties to be a functor.
\end{proof}
\begin{example}
Let $\varphi:\mathcal{A}\rightarrow\mathcal{B}$ be a morphism of filtered bundles of degree $1$, i.e. an affine bundle morphism. Then the linearisation $\pLinr(\phi)$ is just the linear part of $\phi$, i.e. the morphism $\sv(\phi):\sv(\mathcal{A})\to\sv(\mathcal{B})$ of the model vector bundles.
\end{example}
 An observation here  is that since morphism $\phi$ gives rise to a tower of morphisms
$$\xymatrix{
\mathcal{F}_{\{k\}}\ar[d]^{\phi_{\{k \}} := \phi}\ar[r] & \mathcal{F}_{\{k-1\}}\ar[d]^{\phi_{\{k-1\}}}\ar[r] & \cdots \ar[r] & \mathcal{F}_{\{1\}}\ar[d]^{\phi_{\{1\}}}\ar[r] & M\ar[d]_{\phi_0} \\
\mathcal{F}'_{\{k\}}\ar[r] & \mathcal{F}'_{\{k-1\}}\ar[r] & \cdots \ar[r] & \mathcal{F}'_{\{1\}}\ar[r] & M'
}$$
its linearisation gives rise to the tower of filtered-linear morphisms
$$\xymatrix{
\pLinr(\mathcal{F}_{\{k\}})\ar[d]^{\pLinr(\phi)}\ar[r] & \pLinr(\mathcal{F}_{\{k-1\}})\ar[d]^{\pLinr(\phi_{\{k-1\}})}\ar[r] & \cdots \ar[r] & \pLinr(\mathcal{F}_{\{1\}})\ar[d]^{\pLinr(\phi_{\{1\}})}\ar[r]& M\ar[d]_{\phi_0}\\
\pLinr(\mathcal{F}'_{\{k\}})\ar[r] & \pLinr(\mathcal{F}'_{\{k-1\}})\ar[r] & \cdots \ar[r] & \pLinr(\mathcal{F}'_{\{1\}})\ar[r] & M'
}$$

\subsection{Double and multiple filtered structures}
The idea of a double filtered bundle is clear. We have a polynomial bundle equipped with two independent, but compatible, filtered structures. That is we can find an atlas consisting of local coordinates of the form
$$(x^{a},~ Z^{A}_{(u,v)}),$$
where we understand the label
$$(u,v) \in [0,k] \times [0,l] \setminus (0,0).$$
The admissible changes of the fibre  coordinates are of the form
\begin{equation}\label{eqn:transDouble}
Z^{A'}_{(u,v)} = \sum\limits_{\substack{u_{1}+ u_{2} + \cdots + u_{n} \leq u\\  v_{1}+ v_{2} + \cdots + v_{n} \leq v} }  \frac{1}{n!} Z^{B_{1}}_{(u_{1}, v_{1})}Z^{B_{2}}_{(u_{2}, v_{2})} \cdots Z^{B_{n}}_{(u_{n}, v_{n})} T_{B_{n} \cdots B_{2} B_{1}}^{\hspace{35pt} A'}(x).
\end{equation}
We will adopt the notation $\mathcal{F}_{\{k,l \}}$ for a general double filtered bundle of bi-degree $(k,l)$. The order of the bi-degree is important and taken as an intrinsic part of the definition of double filtered bundle.\par
From the coordinate transformations (\ref{eqn:transDouble}) we arrive at the following observation which generalises the tower  (\ref{eqn:tower}), as well as Proposition \ref{prop:Ttower} and Proposition \ref{prop:Vtower}.
\begin{proposition}\label{prop:Doubletower}
Given an arbitrary double filtered bundle, $\mathcal{F}_{\{k,l\}}$ there exists a grid of affine fibre bundle structures
\begin{center}
\leavevmode
\resizebox{9cm}{6cm}{
\begin{xy}
(0,100)*+{\mathcal F_{\{k,l \}}}="a1"; (30,100)*+{\mathcal F_{\{k,l-1 \}}}="a2";
(60,100)*+{\cdots}="a3"; (90,100)*+{\mathcal F_{\{k,1 \}}}="a4";
(120,100)*+{\mathcal F_{\{k,0 \}}}="a5";
(0,80)*+{\mathcal F_{\{k-1,l \}}}="b1"; (30,80)*+{\mathcal F_{\{k-1,l-1 \}}}="b2";
(60,80)*+{\cdots}="b3"; (90,80)*+{\mathcal F_{\{k-1,1 \}}}="b4";
(120,80)*+{\mathcal F_{\{k-1,0 \}}}="b5";
(0,60)*+{\vdots}="c1"; (30,60)*+{\vdots}="c2";
(60,60)*+{\vdots}="c3"; (90,60)*+{\vdots}="c4";
(120,60)*+{\vdots}="c5";
(0,40)*+{\mathcal F_{\{1,l \}}}="d1"; (30,40)*+{\mathcal F_{\{1,l-1 \}}}="d2";
(60,40)*+{\cdots}="d3"; (90,40)*+{\mathcal F_{\{1,1 \}}}="d4";
(120,40)*+{\mathcal F_{\{1,0 \}}}="d5";
(0,20)*+{\mathcal F_{\{0,l \}}}="e1"; (30,20)*+{\mathcal F_{\{0,l-1 \}}}="e2";
(60,20)*+{\cdots}="e3"; (90,20)*+{\mathcal F_{\{0,1 \}}}="e4";
(120,20)*+{\mathcal F_{\{0,0 \}}}="e5";;%
{\ar "a1";"a2"};{\ar "a2";"a3"};{\ar "a3";"a4"};{\ar "a4";"a5"};
{\ar "b1";"b2"};{\ar "b2";"b3"};{\ar "b3";"b4"};{\ar "b4";"b5"};
{\ar "d1";"d2"};{\ar "d2";"d3"};{\ar "d3";"d4"};{\ar "d4";"d5"};
{\ar "e1";"e2"};{\ar "e2";"e3"};{\ar "e3";"e4"};{\ar "e4";"e5"};
{\ar "a1";"b1"};{\ar "a2";"b2"};{\ar "a3";"b3"};{\ar "a4";"b4"}; {\ar "a5";"b5"};
{\ar "b1";"c1"};{\ar "b2";"c2"};{\ar "b3";"c3"};{\ar "b4";"c4"}; {\ar "b5";"c5"};
{\ar "c1";"d1"};{\ar "c2";"d2"};{\ar "c3";"d3"};{\ar "c4";"d4"}; {\ar "c5";"d5"};
{\ar "d1";"e1"};{\ar "d2";"e2"};{\ar "d3";"e3"};{\ar "d4";"e4"}; {\ar "d5";"e5"};
\end{xy}
}
\end{center}
\end{proposition}
\begin{example}
Double affine bundle $\mathbf{A}$ in the sense of \cite{Grabowski:2016} is a double filtered bundle with bi-degree $(1,1)$. Let  $\pLinr_i$, $i=1,2$, denote linearisations with respect to the first and the second affine (filtered) structure. The bundles $\pLinr_1(\mathbf{A})$ and $\pLinr_2(\mathbf{A})$ are both filtered-linear bundles isomorphic to ${\sv}_1(\mathbf{A})$ and
${\sv}_2(\mathbf{A})$ respectively. Applying then linearisation functor to the remaining filtered structure we get two canonically isomorphic double vector bundles $\pLinr({\sv}_1(\mathbf{A}))\simeq \pLinr(\sv_2(\mathbf{A}))$.
\end{example}
\begin{example}
Any filtered-linear bundle can be considered as a double filtered bundle with \newline $Z^{A}_{(u,v)} := (X^{I}_{(w,0)}, ~  Y^{\Sigma}_{(u,1)})$.
\end{example}
A nice construction which provides a huge list of examples of double filtered bundles is the following theorem.
\begin{theorem} For any filtered bundle  $\cF_{\{l\}}\to M$ of degree $l$, the jet bundle $\sJ^k(\cF_{\{l\}})$ is canonically a double filtered bundle of bi-degree $(l,k)$.  The first bundle structure is over $M$, the second is over $\cF_{\{l\}}$.
\end{theorem}
\begin{proof}
Let us introduce coordinates $(x^a, Y^I_w)$ adapted to the structure of filtered bundle in $\cF_{\{l\}}$. It means that the admissible coordinate transformations are of the form
\begin{eqnarray*}
x^{a'} &=& x^{a'}(x), \\
Y_{w}^{I'} &=&  \sum_{w_{1} + \cdots + w_{n}\leq w} \frac{1}{n!} Y^{J_{1}}_{w_{1}} Y^{J_{2}}_{w_{2}} \cdots Y^{J_{n}}_{w_{n}}T_{J_{n} \cdots J_{2} J_{1}}^{\hspace{30pt} I'}(x). \\
\end{eqnarray*}
The most straightforward way to prove that $\sJ^k(\cF_{\{l\}})$ is a double filtered bundle of bi-degree $(k,l)$ would be to just write down the transformation rules for jet coordinates and show explicitly that they respect the bi-weight. For the first jet it  follows from the Leibniz rule for the partial derivative $Y_w^I\mapsto Y^I_{w,a}$,
\begin{eqnarray*}
Y^{I'}_{w; a'} &=&\sum_{w_{1} + \cdots + w_{n}\leq w} \frac{1}{n!} Y^{J_{1}}_{w_{1}} Y^{J_{2}}_{w_{2}} \cdots Y^{J_{n}}_{w_{n}}
\frac{T_{J_{n} \cdots J_{2} J_{1}}^{\hspace{30pt} I'}}{\partial x^b}\frac{\partial x^b}{\partial x^{a'}} \\
&+& \sum_{w_{1} + \cdots + w_{n}\leq w} \frac{1}{(n-1)!} Y^{J_{1}}_{w_{1};c} Y^{J_{2}}_{w_{2}} \cdots Y^{J_{n}}_{w_{n}}
T_{J_{n} \cdots J_{2} J_{1}}^{\hspace{30pt} I'}\frac{\partial x^c}{\partial x^{a'}}\,.
\end{eqnarray*}
We can see in the above equation that both weights are conserved, the first weight is at most $w$ in both components while the jet weight is zero in first and one in the second component.\par
For higher jets the transformation rules become long, therefore we have to change the notation. Jet coordinates will be denoted by $Y^I_{w;\beta}$ where $\beta$ is the multi-index of length $m=\dim M$. The order of the jet equals $|\beta|=\beta^1+\cdots\beta^m$ and each $\beta^a$ denotes the number of derivatives with respect to $x^a$. With this notation the transformation rule for jet coordinates will be the following
$$
Y^{I'}_{w; \beta'}=\sum_{w_{1} + \cdots + w_{n}\leq w}
\sum_{\begin{array}{c}\scriptstyle\alpha_1+\cdots+\alpha_n=\alpha,\\ \scriptstyle |\alpha|\leq|\beta'|\end{array}}
Y^{J_{1}}_{w_{1};\alpha_1} Y^{J_{2}}_{w_{2};\alpha_2} \cdots Y^{J_{n}}_{w_{n};\alpha_n}\;S_{J_{n} \cdots J_{2} J_{1};\beta'}^{ I'\alpha_1\cdots\alpha_n}(x).
$$
The condition $|\alpha|\leq |\beta'|$ means that the jet weight is conserved, while the condition $w_{1} + \cdots + w_{n}\leq w$ assures that the first weight is conserved.
\end{proof}
The operation of jet prolongation may be applied also to morphisms with usual limitations that have nothing to do with the filtered structure: one can apply jet prolongation when it can be assured that a morphism maps sections into sections, i.e. when the induced morphism of base manifolds is a diffeomorphism.
\begin{theorem}
Let $\phi:\mathcal{F}_{\{l\}}\rightarrow \mathcal{F}'_{\{m\}}$ be a morphism of filtered bundles  covering a diffeomorphism on the bases, then its jet prolongation $\sJ^k\phi$ is a morphism of double filtered bundles
$\sJ^k\mathcal{F}_{\{l\}}$ and $\sJ^k\mathcal{F}'_{\{m\}}$ of bidegrees $(l,k)$ and $(m,k)$ respectively.
\end{theorem}
\begin{proof}
We start with the special case of a morphism over the identity on a base manifolds. It means in particular that both filtered bundles are over the same base manifold. The filtered property  is clear from the coordinate expression. Let $\phi: \mathcal{F}_{\{l\}}\rightarrow \mathcal{F}'_{\{m\}}$ be a morphism of filtered bundles over the identity. Using coordinates $(x^a, X^I_w)$ in $\mathcal{F}_{\{l\}}$ and $(x^a, Y^P_u)$ in $\mathcal{F}'_{\{m\}}$ we can express $\phi$ locally
$$
Y^P_u\circ\phi =\sum_{w_{1} + \cdots + w_{n}\leq u}\frac{1}{n!} X^{J_{1}}_{w_{1}} X^{J_{2}}_{w_{2}} \cdots X^{J_{n}}_{w_{n}}\Phi_{J_{n} \cdots J_{2} J_{1}}^{\hspace{30pt} P}(x)\,.  \\
$$
For jet coordinates of order $r\leq k$ we get the expression
\begin{multline*}
Y^P_{(u;\alpha)}\circ\sJ^k\phi=
\sum_{w_{1} + \cdots + w_{n}\leq u}\frac{1}{n!}
\sum_{\beta_1+\cdots+\beta_n+\beta_{n+1}=\alpha}
X^{J_{1}}_{(w_{1};\beta_1)} X^{J_{2}}_{(w_{2};\beta_1)} \cdots X^{J_{n}}_{(w_{n};\beta_n)}
\frac{\partial^{|\beta_{n+1}|}\Phi_{J_{n} \cdots J_{2} J_{1}}^{\hspace{30pt} P}(x)}{\partial x^{\beta_{n+1}}}\,.
\end{multline*}
Both weights in the above expressions are conserved.\par
For more general case of a morphism over a diffeomorphism we shall have local expressions for $\phi$ (using base cordinates $(x^a)$ in $M$ and $(y^b)$ in $M'$)
\begin{eqnarray*}
  y^a\circ\phi &=& \phi^a(x)\,, \\
  Y^P_u\circ\phi &=&\sum_{w_{1} + \cdots + w_{n}\leq u}\frac{1}{n!} X^{J_{1}}_{w_{1}} X^{J_{2}}_{w_{2}} \cdots X^{J_{n}}_{w_{n}}\Phi_{J_{n} \cdots J_{2} J_{1}}^{\hspace{30pt} P}(x)\,.
\end{eqnarray*}
The jet coordinates transform in more complicated way
$$
Y^{P}_{(u; \beta)}=\sum_{w_{1} + \cdots + w_{n}\leq u}
\sum_{\begin{array}{c}\scriptstyle\alpha_1+\cdots+\alpha_n=\alpha,\\ \scriptstyle |\alpha|\leq|\beta|\end{array}}
X^{J_{1}}_{(w_{1};\alpha_1)} X^{J_{2}}_{(w_{2};\alpha_2)} \cdots X^{J_{n}}_{(w_{n};\alpha_n)}\;\Psi_{J_{n} \cdots J_{2} J_{1};\beta}^{ P\alpha_1\cdots\alpha_n}(x)\,,
$$
where $\Psi_{J_{n} \cdots J_{2} J_{1};\beta}^{ P\alpha_1\cdots\alpha_n}(x)$ is an appropriate sum of derivatives of functions $\Phi_{J_{n} \cdots J_{2} J_{1}}^{\hspace{30pt} P}(x)$
multiplied by matrix elements of the derivative of the inverse of base diffeomorphism. Again it is clear from the local expressions that both weights are conserved.
\end{proof}
From the general theory of jet bundles we know that jet prolongations of bundle morphism satisfy the conditions $\sJ^k(\phi\circ\psi)=\sJ^k(\phi)\circ \sJ^k(\psi)$ and $\sJ^k(\mathrm{id_{\mathcal{F}}})=\mathrm{id}_{\sJ^k\mathcal{F}}$.
\begin{example}
The first jet bundle of an affine bundle, $\sJ^1(\mathcal{F}_{\{1\}})$, naturally comes with the structure of a double filtered bundle.  More specifically, we can employ (standard) local coordinates
$$(x^{a},~ Y^{\alpha}, ~Y_{b}^{\beta}),$$
which we can assign bi-weight $(0,0),~ (1,0)$ and $(1,1)$ respectively. The admissible changes of coordinates are
\begin{align*}
& x^{a'} = x^{a'}(x),    \hspace{130pt} Y^{\alpha'} = Y^{\beta}T_{\beta}^{\:\:\: \alpha'}(x) + T^{\alpha'}(x), &\\
& Y^{\alpha'}_{a'} = \left(\frac{\partial x^{b}}{\partial x^{a'}}\right) Y_{b}^{\beta} T_{\beta}^{\:\:\: \alpha'}(x) +  \left(\frac{\partial x^{b}}{\partial x^{a'}}\right) Y^{\beta} \frac{\partial T_{\beta}^{\:\:\: \alpha'}}{\partial x^{b}}(x) +  \left(\frac{\partial x^{b}}{\partial x^{a'}}\right) \frac{\partial T^{ \alpha'}}{\partial x^{b}}(x).
&
\end{align*}
Clearly, these coordinate transformations are filtered.
\end{example}
\begin{example}
Let $E \rightarrow M$ be a fibre bundle in the category of manifolds. The the r-th order jets of sections $\sJ^{r}(\sJ^{k}E)$ of the jet bundle $\sJ^{k} E\rightarrow M$ naturally comes with the structure of a double filtered bundle of bi-degree $(k,r)$. More specifically,  $\sJ^1(\sJ^1 E)$ is a double filtered bundle of bi-degree $(1,1)$ and can be equipped with natural local coordinates
$$ (\underbrace{x^{a}}_{(0,0)},~ \underbrace{y^{\alpha}}_{(0,0)},~ \underbrace{X_{b}^{\beta}}_{(1,0)},~ \underbrace{Y_{c}^{\gamma}}_{(0,1)} , ~ \underbrace{Z_{de}^{\delta}}_{(1,1)} ).$$
The admissible changes of coordinates are of the form
\begin{align*}
& x^{a'} = x^{a'}(x),\\
& y^{\alpha'} = y^{\alpha'}(x,y),\\
& X^{\beta'}_{b'} = \left(  \frac{\partial x^{c}}{\partial x^{b'}}\right)\left(\frac{\partial}{\partial x^{c}} + X_{c}^{\gamma} \frac{\partial}{\partial y^{\gamma}}\right)y^{\beta'}(x,y),\\
& Y^{\gamma'}_{c'} = \left(  \frac{\partial x^{d}}{\partial x^{c'}}\right)\left(\frac{\partial}{\partial x^{d}} + Y_{d}^{\delta} \frac{\partial}{\partial y^{\delta}}\right)y^{\gamma'}(x,y),\\
& Z^{\delta'}_{d'e'} = \left(  \frac{\partial x^{c}}{\partial x^{d'}}\right)\left(\frac{\partial}{\partial x^{c}} + Y_{c}^{\gamma} \frac{\partial}{\partial y^{\gamma}} + Z^{\gamma}_{bc} \frac{\partial}{\partial X_{b}^{\gamma}}\right)X^{\delta'}_{e'}(x,y,X),\\
\end{align*}
which by quick inspection are double filtered.
\end{example}
\begin{example}
Given any filtered bundle the jet manifold $\sJ_{0}(\mathbb{R}^{p} , \mathcal{F}_{\{k\}})$ is canonically a double filtered bundle. More specifically, we can employ naturally induced coordinates
$$(x^{a},~ X^{I}_{(w,0)},~ x^{b}_{(0,\lambda)}, ~ X^{J}_{(w,\lambda)}),$$
where $0 < w \leq k$ and $0 < \lambda \leq p$. The transformation laws for the coordinates follows from example (\ref{exp:pplane}), specifically we have
\begin{eqnarray*}
X^{I'}_{(w,\lambda)} &=& \sum_{w_{1} + \cdots + w_{n}\leq w} \frac{1}{(n-1)!}  X^{J_{1}}_{(w_{1},\lambda)} X^{J_{2}}_{(w_{2},0)} \cdots X^{J_{n}}_{(w_{n},0)}T_{J_{n} \cdots J_{2} J_{1}}^{\hspace{30pt} I'}\\
 &+& \sum_{w_{1} + \cdots + w_{n}\leq w} \frac{1}{n!} X^{J_{1}}_{(w_{1},0)} X^{J_{2}}_{(w_{2},0)} \cdots X^{J_{n}}_{(w_{n},0)}   x^{a}_{(\lambda,0)} \frac{\partial }{\partial x^{a}}T_{J_{n} \cdots J_{2} J_{1}}^{\hspace{30pt} I'}.
\end{eqnarray*}
Clearly we have a double filtered bundle structure, albeit one of the structures is split. The choice of coordinates here induces the obvious (partial) splitting
$$\sJ_{0}(\mathbb{R}^{p} , \mathcal{F}_{\{k\}}) \simeq \bigoplus_{p}\sT \mathcal{F}_{\{k\}}.$$
\end{example}
The notion of  higher $n$-tuple filtered bundles is clear in terms of local coordinates. That is we have a polynomial bundle equipped with local coordinates that can simultaneously be assigned an $n$-tuple weight, and the changes of local coordinates are filtered with respect to each weight. As a matter of notation we will denote a general $n$-tuple filtered bundle as $\mathcal{F}_{\{k_{1}, k_{2}, \cdots , k_{n} \}}$ to indicate the  assignment of the multi-weight to the local coordinates.
\begin{example}
Double vector bundles (see \cite{Pradines:1974}) are examples of double filtered bundles. Similarly,  $n$-tuple vector bundles and $n$-tuple graded bundles are examples of $n$-tuple filtered bundles.
\end{example}
\begin{example}
Repeated jets of sections of a vector bundle  $\sJ^{k_1}{\sJ}^{k_2} \cdots {\sJ}^{k_n}(E)$ (cf. Remark \ref{re1}) naturally comes with the structure of an $n$-tuple filtered bundle.
\end{example}
\begin{example}
Repeated application of the linearisation functor to a filtered bundle $\mathcal{F}$ of degree $k$ leads to a $k$-tuple
vector bundle $\Linr(\mathcal{F})=\pLinr(\pLinr(\cdots\pLinr(\mathcal{F})\cdots))$, which we can view as a $k$-tuple filtered bundle.  In this way, we obtain a \emph{total linearisation (polarisation) functor} as in \cite{Bruce:2016}.
\end{example}
\begin{proposition}
There is an equivalence  between the category of $n$-tuple filtered bundles of the form $\mathcal{F}_{\{1,1,\cdots, 1 \}}$ and  the category of $n$-affine bundles (cf. \cite{Grabowski:2010}). In particular, filtered bundles of the form $\mathcal{F}_{\{1,1 \}}$ are precisely double affine bundles.
\end{proposition}
\begin{proof}
Let us just concentrate on the $(1,1)$ case as this is illustrative enough to establish the general case. Let us employ local coordinates $(x^{a},~ Y^{I}_{(1,0)},~ Z^{\Sigma}_{(0,1)}, ~ W^{A}_{(1,1)})$ on a general double filtered bundle $\mathcal{F}_{\{1,1 \}}$. The admissible changes of local coordinates are of the form
\begin{align*}
& x^{a'} = x^{a'}(x),\\
& Y^{I'}_{(1,0)} = Y^{J}_{(1,0)}T_{J}^{\:\:\: I'}(x) + T^{I'}(x),\\
& Z^{\Sigma'}_{(0,1)} = Z^{\Upsilon}_{(0,1)}T_{\Upsilon}^{\:\:\: \Sigma'}(x) + T^{\Sigma'}(x),\\
& W^{A'}_{(1,1)} = W^{B}_{(1,1)}T_{B}^{\:\:\: A'}(x) + Y^{I}_{(1,0)}Z^{\Sigma}_{(0,1)}T_{\Sigma I}^{\:\:\:\: A'}(x) + Y^{J}_{(1,0)}T_{J}^{\:\:\: A'}(x) + Z^{\Sigma}_{(0,1)}T_{\Sigma}^{\:\:\: A'}(x) + T^{A'}(x),
\end{align*}
which in accordance with \cite[Theorem 2.1]{Grabowski:2010} provide an adapted coordinate system on a double affine bundle. Morphisms of double affine bundles are `affine' and thus in one-to-one correspondence with respective morphisms of double filtered bundles.
\end{proof}
\begin{remark}
An observation here is that, following example   $\sJ^1\mathcal{F}_{\{1\}}$ is a double affine bundle.
\end{remark}
As a matter of formality, lett us denote the category of $n$-tuple filtered bundles by $\catname{FilB}^n$. More specifically, when each filtered bundle structure is of degree $k_i$ ($1 \leq i \leq n$) we will denote the corresponding full subcategory as $\catname{FilB}^n[k_1, k_2, \cdots , k_n]$. Morphisms in these categories are morphisms as polynomial bundles that preserve the filtered structures, in local coordinates this meaning is clear. \par
By passing to a total weight, i.e., sum the individual weights,  any $n$-tuple filtered bundle can be considered as a filtered bundle -- coordinate changes that respect the multi-weight also respect the total weight. It is easy to see that this process is functorial: again anything that respects the multi-weight also respect the total weight.  These considerations lead to the following observation.
\begin{proposition}
There is a functor
$$\textnormal{Totw}: \catname{FilB}^n[k_1, k_2, \cdots , k_n] \longrightarrow \catname{FilB}[k_1+k_2 + \cdots + k_n],$$
that is constructed by passing from the assignment of  a muilt-weight to  the assignment of a total weight of the local coordinates in any adapted atlas.
\end{proposition}
 For example, any double affine bundle can be considered as a filtered bundle of degree $2$. A similar observation can be made of a double vector bundle, which can always be considered as a graded bundle of degree $2$.  


\vskip1cm

\noindent Andrew James Bruce\\
\emph{Mathematics Research Unit, University of Luxembourg,}\\ {\small Maison du Nombre 6, avenue de la Fonte,
L-4364, Esch-sur-Alzette,  Luxembourg}\\ {\tt andrewjamesbruce@googlemail.com}\\

\noindent Katarzyna Grabowska\\
\emph{Faculty of Physics,
                University of Warsaw} \\
               {\small Pasteura 5, 02-093 Warszawa, Poland} \\
                 {\tt konieczn@fuw.edu.pl} \\

\noindent Janusz Grabowski\\\emph{Institute of Mathematics, Polish Academy of Sciences}\\{\small \'Sniadeckich 8, 00-656 Warszawa,
Poland}\\{\tt jagrab@impan.pl}


\begin{thebibliography}{10}
\begin{small}

\bibitem{Bertram:2014}
W.~Bertram \&  A.~Souvay,
\newblock{A general construction of Weil functors,}
\newblock{\emph{Cah. Topol. G\'{e}om. Diff\'{e}r. Cat\'{e}g.} \textbf{55}(4) (2014) 267--313 .}

\bibitem{Bruce:2014}
A.J.~Bruce, K.~Grabowska \& J.~Grabowski,
\newblock{Linear duals of graded bundles and higher analogues of (Lie) algebroids,}
\newblock{{\emph{J. Geom. Phys.}}, \textbf{101} (2016), 71--99.}

\bibitem{Bruce:2014b}
A.J.~Bruce, K.~Grabowska \& J.~Grabowski,
\newblock{Higher order mechanics on graded bundles,}
\newblock{{\emph{J. Phys. A}} \textbf{48} (2015), 205203, 32pp.} 

\bibitem{Bruce:2016}
A.J.~Bruce, J.~Grabowski~J. \& M.~Rotkiewicz,
\newblock{Polarisation of graded bundles,}
\newblock{{\emph{SIGMA}} \textbf{12} (2016), 106, 30pp.}

\bibitem{Grabowska:2005}
K.~Grabowska, J.~Grabowski \& P. ~Urba\'{n}ski,
\newblock{Frame-independent mechanics: geometry on affine bundles,}
\newblock{ \emph{Travaux math\'ematiques. Fasc. XVI}, 107--120,
Trav. Math., \textbf{16}, \emph{Univ. Luxemb., Luxembourg}, 2005.}

\bibitem{Grabowski:2013}
J.~Grabowski,
\newblock{Graded contact manifolds and contact Courant algebroids,}
\newblock{\emph{J. Geom. Phys.} \textbf{68} (2013), 27--58.}

\bibitem{Grabowski:2016}
J.~Grabowski, M.~J\'{o}\'{z}wikowski, M.~Rotkiewicz,
\newblock{Duality for graded manifolds,}
\newblock{ to appear in \emph{Rep. Math. Phys.},  {arXiv:1610.01888}.}

\bibitem{Grabowski:2009}
J.~Grabowski \& M.~Rotkiewicz,
\newblock{Higher vector bundles and multi-graded symplectic manifolds,}
\newblock{\emph{J. Geom. Phys.} \textbf{59}(9) (2009),  1285--1305.}

\bibitem{Grabowski:2012}
{J.~Grabowski \& M.~Rotkiewicz,}
\newblock{Graded bundles and homogeneity structures,}
\newblock{\emph{J. Geom. Phys.} \textbf{62} (2012), no. 1, 21--36. }

\bibitem{Grabowski:2010}
J.~Grabowski, P. ~Urba\'{n}ski \&  M.~Rotkiewicz,
\newblock{Double affine bundles,}
\newblock{\emph{J. Geom. Phys.} \textbf{60}(4),  581--598 (2010).}

\bibitem{Jozwikowski:2016}
M.~J\'{o}\'{z}wikowski \& M.~Rotkiewicz,
\newblock{A note on actions of some monoids,}
\newblock{\emph{Differential Geom. Appl.} \textbf{47} (2016), 212--245.}

\bibitem{Kolar:1993}
I.~Kol\'{a}\v{r},  P.~Michor \& J.~Slov\'{a}k,
\newblock{\emph{Natural  Operations  in  Differential  Geometry}},
\newblock{Springer, Berlin, 1993.}

\bibitem{Kontsevich:2003}
M.~Kontsevich,
\newblock{Deformation quantization of Poisson manifolds,}
\newblock{ \emph{Lett. Math. Phys.}, \textbf{66}(3) (2003), 157-216,}

\bibitem{Mehta:2006}
R.A.~Mehta,
\newblock{Supergroupoids, double structures, and equivariant cohomology,}
\newblock{ thesis, University of California at Berkeley, 2006,  \texttt{math.DG/06053}.}

\bibitem{Miron:2003}
R.~Miron,
\newblock{\emph{The Geometry of Higher-Order Hamiltonian Spaces. Applications to Hamiltonian Mechanics},}
\newblock{Kluwer Dordrecht, FTPH no 132, 2003.}

\bibitem{Morimoto:2000}
T.~Morinoto,
\newblock{Th\'{e}or\`{e}me d'existence de solutions analytiques pour des syst\`{e}mes d' \'{e}quations aux d\'{e}riv\'{e}es partielles non-lin\'{e}aires avec singularit\'{e}s,}
\newblock{\emph{C.R. Acad. Sci. Paris}. \textbf{321}, s\'{e}rie 1. 1995, 1491--1496.}

\bibitem{Popescu:2005}
P.~Popescu \&  M.~ Popescu,
\newblock{Lagrangians and Hamiltonians on affine bundles and higher order geometry,}
\newblock{\emph{Geometry and topology of manifolds}, 451--469,
Banach Center Publ., \textbf{76}, \emph{Polish Acad. Sci. Inst. Math., Warsaw}, 2007. }

\bibitem{Pradines:1974}
J.~Pradines,
\newblock{Repr\'{e}sentation des jets non holonomes par des morphismes vectoriels doubles soud\'{e}s},
\newblock{\emph{C. R. Acad. Sci. Paris S\'{e}r. A} \textbf{278} (1974) 152---1526.}

\bibitem{Roytenberg:2002}
D.~Roytenberg,
\newblock{ On the structure of graded symplectic supermanifolds and Courant algebroids,}
\newblock{in \emph{Quantization, Poisson brackets and beyond} (Manchester, 2001), 169--185, \emph{Contemp. Math.} \textbf{315},
Amer. Math. Soc., Providence, RI, 2002.}

\bibitem{Severa:2005}
P.~\v{S}evera,
\newblock{Some title containing the words ``homotopy'' and ``symplectic'', e.g. this one,}
\newblock{ \emph{Travaux Math\'{e}matiques. Fasc. XVI}, 121--137,
Trav. Math., XVI, Univ. Luxemb., Luxembourg, 2005.}

\bibitem{Voronov:2001qf}
Th.Th.~Voronov,
\newblock {Graded manifolds and {D}rinfeld doubles for {L}ie bialgebroids}.
\newblock{ In Quantization, Poisson Brackets and Beyond, volume \textbf{315} of
 \emph{ Contemp. Math}., pages 131--168. Amer. Math. Soc., Providence, RI, 2002.}


 \bibitem{Voronov:2010}
Th.Th.~Voronov,
\newblock{ Q-manifolds and higher analogs of Lie algebroids,}
\newblock{ In XXIX Workshop on Geometric Methods in Physics, AIP CP 1307, pages 191--202. Amer. Inst. Phys., Melville, NY, 2010.}




\end{small}
\end{thebibliography}
\end{document}